\documentclass[12pt]{amsart}
\usepackage[top=2cm, bottom=2cm, left=2cm, right=2cm]{geometry}
\usepackage{hyperref}
\usepackage{graphicx,rotating, pb-diagram}
\usepackage[all,cmtip]{xy}

\author[A. D. Smith]{Abraham D. Smith}
\address{McGill University, Department of Mathematics and Statistics\\
         Montreal, Quebec  H3A 2K6\\
         Canada}
\email{\href{mailto:adsmith@msri.org}{adsmith@msri.org}}
\urladdr{\href{http://www.math.mcgill.ca/adsmith}%
         {http://www.math.mcgill.ca/adsmith}}

\thanks{This work is supported by an NSF All-Institutes
Postdoctoral Fellowship administered by the Mathematical Sciences Research
Institute through its core grant DMS-0441170. The author is hosted by the
Department of Mathematics and Statistics at McGill University.}

\title[A Geometry for PDEs, I]
{A Geometry for Second-Order PDEs and their Integrability, Part I}
\date{\today}

\newtheorem{thm}{Theorem}[section]
\newtheorem{lemma}[thm]{Lemma}
\newtheorem{cor}[thm]{Corollary}
\newtheorem{defn}[thm]{Definition}
\newcommand{\lhk}{\mathbin{\hbox{\vrule height1.4pt width4pt depth-1pt
\vrule height4pt width0.4pt depth-1pt}}} 

\newcommand{\pair}[1]{\ensuremath{\left\langle #1 \right\rangle}}
\DeclareMathOperator{\tr}{tr}

\DeclareMathOperator{\ver}{\mathrm{ver}}

\begin{document}

\begin{abstract}
For the purpose of understanding second-order scalar PDEs and their hydrodynamic
integrability, we introduce G-structures that
are induced on hypersurfaces of the space of symmetric matrices (interpreted
as the fiber of second-order jet space) and are defined by non-degenerate scalar second-order-only
(Hessian) PDEs in any number of variables.  The fiber group is a conformal orthogonal group
that acts on the space of independent variables, and it is a subgroup of the conformal
orthogonal group for a semi-Riemannian metric that exists on the PDE.   These
G-structures are automatically compatible with the definition of hydrodynamic
integrability, so they allow contact-invariant analysis of integrability via
moving frames and the Cartan--K\"ahler theorem.  They directly generalize the
GL(2)-structures that arise in the case of Hessian hyperbolic equations in
three variables as well as several related geometries that appear in the
literature on hydrodynamic integrability.  Part I primarily discusses the motivation, the definition, and the solution
to the equivalence problem, and Part II will discuss integrability in detail.
\end{abstract}

\maketitle
\tableofcontents

\pagebreak

\section*{Introduction}
The primary motivation for this article and its sequel is a geometric
classification of second-order scalar partial differential equations [PDEs] in
any number of variables, \[ F\left(x^1, \ldots, x^n, z, \frac{\partial
z}{\partial x^1} , \ldots, \frac{\partial z}{\partial x^n}, \frac{\partial^2
z}{\partial x^1 x^1}, \ldots, \frac{\partial^2 z}{\partial x^n
x^n}\right)=0.\] Such a classification should describe the intrinsic structure
of the PDEs, meaning that it should be invariant under any contact transformation of
the PDE.  Such a classification is especially interesting if it also
highlights the integrable PDEs as those with certain explicit conditions on
their defining invariants.

A classification of all scalar second-order PDEs under the entire pseudo-group
of contact transformations is infinitely far beyond the scope of this article.  Instead,
this article focuses on the special case of Hessian PDEs, which are those of the form \[
F\left(\frac{\partial^2 z}{\partial x^1 x^1}, \ldots, \frac{\partial^2
z}{\partial x^n x^n}\right)=0.\] A Hessian PDE may be interpreted as a
hypersurface $F^{-1}(0)$ in $\mathrm{Sym}^2(\mathbb{R}^n)$.  The class of
Hessian PDEs is important for two reasons.
First, this class of PDEs contains many interesting examples, such as the wave
equation, the first flow of the dispersionless Kadomtsev--Petviashvili hierarchy, and the
symplectic Monge--Amp\`ere equations.  Second, a recent theorem of Dennis The
shows that any classification of Hessian PDEs with non-degenerate symbol up to
the standard action of the conformal symplectic group yields a (somewhat
coarse) classification of all second-order PDEs with non-degenerate symbol up
to contact transformation \cite[Section 2.3]{The2010}.  This is because the
contact transformations restrict on each fiber of second-order jet space to
give the conformal symplectic transformations, and Hessian PDEs may be
considered as the intersection of a general second-order PDE with any
particular fiber of second-order jet space.

For functions $z$ over the real numbers, hyperbolic PDEs---those with
leading symbol of signature $(n-1,1)$---are particularly interesting, since these have the
least-degenerate characteristics and are thus most relevant to any reasonable notion of
integrability.  Aside from occasionally superfluous scaling factors, most
results here also apply to the complex case with arbitrary non-degenerate symbol.

This article fits within a constellation of recent results on hyperbolic
Hessian PDEs.
In the case $n=2$, a highly detailed classification of hyperbolic Hessian PDEs
is given by Dennis The using the extrinsic geometry of surfaces embedded in
$\mathrm{Sym}^2(\mathbb{R}^2) \cong \mathbb{R}^3$ \cite{The2008}
\cite{The2010}.  In the case $n=3$, an intrinsic classification of integrable
hyperbolic Hessian PDEs is given by the leaves of a singular foliation of
$\mathbb{R}^9$, as found by the present author using intrinsic $GL(2)$
geometry \cite{Smith2009a}.  The project \cite{Smith2009a} was inspired by
earlier extrinsic work of Ferapontov \emph{et al.} that links the Veronese cone to the notion of hydrodynamic integrability \cite{Ferapontov2009}.   In the case $n=4$, Doubrov and Ferapontov classify
the symplectic Monge--Amp\`ere equations, which form an important subclass of
integrable hyperbolic Hessian PDEs \cite{Doubrov2009}.  Another 
related study is \cite{Alekseevsky2010}, which uses the Veronese cone to
investigate the Cauchy problem for a class of hyperbolic Hessian
PDEs introduced by Goursat.  Also, in the case of arbitrary $n$, quasilinear 
second-order PDEs (not Hessian) are analyzed in \cite{Burovskiy2008} using an extrinsic $SL(n+1)$
geometry that is related to the geometry seen here for Hessian PDEs by the
sort of fiber-wise coarse classification mentioned above.
In this project, the
geometry of Hessian PDEs with non-degenerate symbol is considered for general $n$
in both the integrable and non-integrable cases.
Thus, this project offers a unified view of the geometries that provide these
recent results.

In Part I, we introduce a particular $G$-structure (see Definition~\ref{maindef}) that is
induced on non-degenerate Hessian PDEs, interpreted as hypersurfaces of
$\mathrm{Sym}^2(\mathbb{R}^n)$.  These $G$-structures admit very simple global
structure equations that have only finitely many local invariants (see
Theorem~\ref{thmembedstr}).   These
structure equations can be readily computed for any $n$, but attention here is
restricted to the cases $n=2$, $n=3$, and $n=4$, which are most relevant for physical
examples.  This geometry is not strictly-speaking ``new,'' as it has been encountered in various guises and special
cases in each of the references given above, but no previous work takes
complete advantage of this geometry as a concept that is well-defined on any
non-degenerate Hessian PDE in any number of variables and
admits invariant analysis using the method of equivalence.  The main theorem
is a complete set of structure equations that indicate how to classify all
such PDEs.  
In implementing the method of equivalence here, we attempt to keep the representation theory as elementary as possible,
focusing only on a rough decomposition of vector spaces into submodules under
the action of specific orthogonal
groups.  This approach has the advantage of keeping the structure equations in a form that can
be easily entered into computer algebra systems without any understanding of spinor representations or
Clebsch--Gordan decompositions.  However, finer detail would result from
application of such knowledge for a specific group $SO(n-k,k)$.
In the sequel\footnote{Part I and Part II will be separate as preprints on the
arXiv, but it is likely that they will be unified for journal submission.},
Part II, the relationship between these local invariants and the property of hydrodynamic
integrability will be explored more deeply.

I wish to express my thanks to Niky Kamran, for his encouragement of
my pursuit of this geometry, to Dennis The for many stimulating discussions
of PDE and jet-space geometry, and to Francis Valiquette for expanding my perspective on the
equivalence problem.

\textbf{Use of Indices:}  Because the geometry in this article uses a representation
of $SO(n)$ other than the standard representation, the conventional up/down
index summation notation is \emph{not} used, as it would lead to confusion.  
Sums are indicated explicitly with $\sum$.   Whether an index is up or down
for a particular object depends only on aesthetics and convenience, and the
reader must keep track of whether a particular object is co- or
contra-variant.    Moreover, throughout this article, juxtaposition without
indices 
\emph{always} means matrix multiplication in $\mathfrak{gl}(n)$.  Thus, if $a$
and $b$ are differential forms valued in $\mathfrak{gl}(n)$, then $a\wedge b$
indicates the $\mathfrak{gl}(n)$-valued differential form with $(i,j)$ entry
$(a \wedge b)_{ij} =\sum_k a_{ik}\wedge b_{kj}$.  In particular, juxtaposition is never used to
indicate composition of functions; instead, the notations $f \circ g$ and
$f(g)$ are used as appropriate.  The tensor, symmetric, and skew products are
denoted by $\otimes$, $\odot$, and $\wedge$, respectively.  The symbol
$\nabla(x)$ always indicates the covariant derivative of $x$ with respect to a
connection $\theta$ that acts on $x$ through a representation $\rho$, so
$\nabla(x) = \mathrm{d}x + \rho_\theta(x)$.  Finally, $V \otimes W^*$ is
identified with $\mathrm{Hom}(W,V)$.

\section{Background and Definitions}
\label{back}
This section discusses the conformal symplectic group over
$\mathrm{Sym}^2(\mathbb{R}^n)$, which is the ambient structure that induces 
the intrinsic geometry on a Hessian PDE.

\subsection{The Conformal Symplectic Group}
Let $V = \mathbb{R}^n$, considered as row vectors.  Let $V \odot V =
\mathrm{Sym}^2(V)$ be identified with the space of $n \times n$ symmetric
matrices.
Consider the principal bundle $\mathcal{F}$ of $(V \odot V)$-valued co-frames
over the manifold $\Lambda^o=\mathrm{Sym}^2(V)$.  Obviously, $\Lambda^o$ and $V
\odot V$ are identical, but here $\Lambda^o$ refers to the set of
symmetric matrices regarded as a manifold, which can be seen as
a simply
connected coordinate chart in the Lagrangian Grassmannian $\Lambda =
LG(n,2n)$, and $V \odot V$ refers to the set of symmetric matrices
regarded as a vector space.  The fiber over $U \in
\Lambda^o$ is 
\[ \mathcal{F}_U = \left\{ a : \mathbf{T}_U
\Lambda^o \overset{\sim}\to V \odot V \right\} \cong
GL\left(n(n+1)/2\right).\]

There is a distinguished reduction of this coframe bundle to a $GL(n)$-bundle,
and after prolongation, this distinguished reduction has a total space
isomorphic to the conformal symplectic group.  There are three ways to
see this structure, and all of them are important to understand.

First, the projective space $\mathbb{P}(V \odot V)$ contains a non-degenerate sub-variety, the
Veronese variety, which is defined as the projective variety corresponding to
the image of the Veronese map $ \ver_2: V
\to V \odot V$ by $\ver_2(v) =v^{\top} v$.  In particular, the Veronese cone $\ver_2(V)$, is the
de-projectivized variety comprised of rank-one symmetric matrices, and the
cone is in birational correspondence with $V$.  The action on
$\mathbb{P}(V{\odot}V)$ by symmetries of $\ver_2(\mathbb{P}V)$ is a
representation of $PGL(n)$ given by $[A] \mapsto [ g^{\top} A g ]$ for any
representative $g$ of $[g] \in PGL(n) = GL(n)/(\mathbb{R}I)$.  Here, this representation is
called $POd(PGL(n))$.
For the Veronese cone in the affine space $V \odot V$ over $\mathbb{R}$, an accurate description
of the symmetry group $Od(GL(n))$ is a little more complicated; see
Appendix~\ref{neg}.
One may consider an $Od(GL(n))$ reduction of $\mathcal{F}$.  The reduced
bundle, $\mathcal{F}_{Od(GL(n))}$, has a tautological semi-basic 1-form $\alpha=
(\alpha_{ij}) = (\alpha_{ji}) = \alpha^{\top}$ that takes values in $V \odot
V$, and the flat choice of section
for the frame bundle, $a_{ij} = \mathrm{d}U_{ij}$, defines a pseudo-connection
$\beta = (\beta_{ij})$ valued in $\mathfrak{gl}(n)$.  A study of the
equivalence problem for this bundle (see \cite{Singer1965}) shows that
after one prolongation, the total space of the bundle has the structure
equations of $\mathfrak{sp}(n)$, 
the symplectic algebra on $\mathbb{R}^{2n}$:
\begin{equation}
\begin{split}
\mu &=\begin{pmatrix}
\beta & \gamma\\ \alpha & -\beta^{\top}
\end{pmatrix},\ \text{for}\  \alpha = \alpha^{\top}, \gamma=\gamma^{\top}, \beta \in
\mathfrak{gl}(n)\\
0&=\mathrm{d}\mu + \mu \wedge \mu = 
\begin{cases}
\mathrm{d}\alpha_{ij} + \sum_{k} \alpha_{ik}\wedge\beta_{kj} - \beta_{ki}
\wedge \alpha_{kj}, & \\
\mathrm{d}\beta_{ij} + \sum_{k} \beta_{ik} \wedge \beta_{kj} +
\gamma_{ik}\wedge \alpha_{kj},
& \\
\mathrm{d}\gamma_{ij} + \sum_{k} \beta_{ik} \wedge \gamma_{kj} -
\gamma_{ik}\wedge \beta_{jk}.
\end{cases}
\end{split}
\label{mc}
\end{equation}
Thus, the total space of the prolonged, flat $Od(Gl(n))$-structure is 
the unique simply-connected Lie group that is the open set near the identity
of $Sp(n)$.  Over the reals, when the $\pm 1$ scaling is allowed, the group
completes to become the conformal symplectic group, $CSp(n) \subset
GL(2n,\mathbb{R})$, 
\begin{equation}
CSp(n) 
= \left\{\begin{pmatrix}
B & C \\ 
A & D
\end{pmatrix}, 
A^{\top}B - B^{\top}A =D^{\top}C-C^{\top}D=0, D^{\top}B - C^{\top}A = c I_n
\neq 0
\right\}.
\label{csp}
\end{equation}

Second, consider the action on the space of symmetric bilinear forms, $V^* \odot 
V^*$, that is induced by the standard representation of $GL(n)$ on $V$.  That
is, consider how the coefficients $\Phi_{ij}$ in the equation
$\sum_{ij}\Phi_{ij} v_i w_j =0$ vary when the coordinates $(v_1, \ldots, v_n)$
are transformed by $v \mapsto vg$ on $V$.  This action is 
$Od^*(GL(n))$, given by $\Phi \mapsto \lambda g^{-1} \Phi g^{-1,\top}$.  
Now, any such $\Phi \in V^* \odot V^*$ may be considered as an element of $(V
\odot V)^*$,
so it defines a hyperplane $\Phi^\perp = \{ a  \in V \odot V : \tr(\Phi a) = 0
\} \subset V \odot V$, which is unique up to scale.  By the law of inertia,
the action $Od_g$ is transitive on symmetric matrices of the same signature.
Thus, $Od(GL(n))$ may also be characterized
as the group that acts transitively on the space of
non-degenerate hyperplanes in $V \odot V$ (preserving signature over the
reals), so $\mathcal{F}_{Od(GL(n))}$ is also the bundle of
frame changes that act transitively on the corresponding codimension-one distributions over
$\Lambda^o$.
This is essentially the perspective taken by Cartan in case ``$\alpha$'' of
Theorems XIX and XX of \cite{Cartan1909} and subsequently clarified in \cite{Singer1965}
\cite{Guillemin1966}, and \cite{Shnider1970}.

Finally, here is the tautological description of jet space that appears in
\cite{Yamaguchi1983}.  In this description, zeroth-order jet space is $J^0$, a
trivial bundle over $\mathbb{R}^n$ with fiber $\mathbb{R}$.  The first-order
jet space is $J^1 = Gr_n(\mathbf{T}J^0) = \mathbf{T}^*J^0$, which admits a
tautological 1-form, $\Upsilon_E = E^\perp \circ \pi^1_0$, where  $E^\perp$ is
the annihilator of $E \in Gr(\mathbf{T}J^0)$ and $\pi^1_0$ is the projection
from $J^1$ to $J^0$.  The Pfaff theorem holds on the differential system
generated by $\Upsilon$, so the space of maximal integral elements
$\mathcal{V}_{\text{max}}(\Upsilon) \subset Gr_n (\mathbf{T}J^1)$ is a smooth
bundle over $J^1$.  The total space of this bundle is $J^2$, and each fiber of
the projection $\pi^2_1:J^2 \to J^1$ is
isomorphic to the Lagrangian Grassmannian $\Lambda = \{ E \in Gr_n(2n) :
\sigma|_E = 0\}$ that is associated to the standard symplectic form $\sigma$
on $\mathbb{R}^{2n}$.  Given local coordinates $(x^i,z,p_i,U_{ij})$ with
$U_{ij} = U_{ji}$ on $J^2$ and an independence condition $\mathrm{d}x^1 \wedge
\cdots \wedge\mathrm{d}x^n \neq 0$, the canonical system
generated by $\Upsilon$ may be written as 
\begin{equation}
\begin{split}
\Upsilon &= \mathrm{d}z - \sum_i p_i \mathrm{d}x^i,\ \text{and}\\
\mathrm{d}\Upsilon &= \sum_i \mathrm{d}p_i \wedge \mathrm{d}x^i,\ \text{so} \\
0 &= \mathrm{d}p_i - \sum_j U_{ij} \mathrm{d}x^j.
\end{split}
\end{equation}
A contact transformation is an isomorphism of the bundle $J^2$ that
preserves the canonical system generated by $\Upsilon$ up to scale.  The contact
transformation on $J^2$ restricts to the fiber as an action of the conformal
symplectic group on $LG(n,2n)$.  In local coordinates $(U_{ij})$ for the
fiber, this appears as the $Od(GL(n))$ action on the matrix $U \in \Lambda^o$.

Thus, the conformal symplectic group, which is the prolongation of the
flat $Od(GL(n))$ frame bundle on $\Lambda^o$, is the appropriate setting to
study the properties of Hessian PDEs that are invariant under the family of  
contact transformations that preserve a specific fiber of $J^2$. 
For more detail regarding this fiber-wise action and the associated
notion of constant symplectic invariant, consult Section 2.3 of \cite{The2010}.

\subsection{Hypersurfaces and Hyperbolicity}
Suppose that $F^{-1}(0)$ is a hypersurface in $\Lambda^o$, and only consider
the hypersurface near points where $\mathrm{d}F \neq 0$ so that the implicit
function theorem applies.
Let $a=(a_{ij})$ be
a $(V \odot V)$-valued moving frame on $\Lambda^o$ that is
$Od(GL(n))$-equivalent to the flat section, $\mathrm{d}U_{ij}$, so $a$ is a
section of $\mathcal{F}_{Od(GL(n))}$.  Then $\mathrm{d}F_U = \sum_{ij}
\Phi_{ij}(U)a_{ij}(U)$ for some $\Phi:\Lambda^o \to (V \odot V)^* = V^* \odot
V^*$.  When $\Phi$ is interpreted as a symmetric bilinear form on $V$, namely
$(v,w) \mapsto v \Phi w^{\top}$, it is precisely the leading symbol of the PDE
$F$ as written in the flat coordinates determined by the co-frame $a$.  Under
a coordinate change $g:v \mapsto vg$ for $g \in GL(n)$, the symmetric bilinear
form changes as $\Phi \mapsto g^{-1} \Phi g^{\top,-1}$.  This corresponds to
the $Od_g$ action on the co-frame $a$.

At each point $U \in F^{-1}(0)$, the intersection $a(\ker\mathrm{d}F|_U) \cap
\ver_2(V)$ gives the equation of a quadric in $\mathbb{P}V$, $ 0 = \sum_{ij}
\Phi_{ij}(U) v_i v_j$.  This quadric is non-degenerate if and only if the
matrix $\Phi(U)$ is non-singular.  In the real case, the most interesting case
is the maximal intersection, which occurs when $\Phi(U)$ has signature $(n-1,1)$.
In this case, the hypersurface is called \textbf{hyperbolic} at $U$.  Given the
discussion $J^2$ in the previous section, this matches the traditional notion
of hyperbolicity for PDEs.  Because the signature of $\Phi(U)$ is preserved by
the $Od^*(GL(n))$ action, the signature of non-degenerate $\Phi(U)$ is an
example of a ``constant symplectic invariant'' in the terminology of
\cite{The2010}, so this is an easy way to see that hyperbolicity (or any other
non-degenerate signature) is a contact-invariant property of a real
second-order PDE near generic points in jet space where the highest-order
terms of $F$ have maximal rank. 

Now, consider the subgroup of $Od(GL(n))$ that preserves $\ker
\mathrm{d}F|_U$.  This subgroup must preserve the bilinear form $\Phi(U)$ up to
scale, so it contains $O(n,\Phi(U))=\{ g \in GL(n) : g \Phi(U) g^{\top} = 
\Phi(U) \}$.  The representation theory of this group is
central to the main result, so it is explored in the next section before
proceeding to the local geometry.

\subsection{Infinitesimal Geometry}
\label{infingeom}

Consider $V = \mathbb{R}^n$ as the vector space of row vectors.  Fix a
non-degenerate symmetric bilinear form $\Phi$ on $V$.
Denote the $\Phi$-null cone in $V$ by $\mathcal{N} = \{ v \in V : v \Phi v^{\top} =
0\}$. Let $O(n,\Phi) = \{ g \in GL(n) : g\Phi g^{\top} = \Phi \}$ which has Lie
algebra $\mathfrak{so}(n,\Phi) = \{ X \in \mathfrak{gl}(n) : X\Phi + \Phi
X^{\top} = 0 \}$.  If $\Phi$ has signature $(n-1,1)$, then the Lie group
$O(n,\Phi)$ is isomorphic to $O(n-1,1)$, the 
Lorentz group.  Let $G$ denote the conformal Lorentz group, $G=CO(n,\Phi) = \{ g \in GL(n) : g
\Phi g^{\top} = \lambda \Phi,\ \lambda \neq 0\}$, which has Lie algebra
$\mathfrak{g}=\mathfrak{co}(n,\Phi) = \mathfrak{so}(n,\Phi) +
\mathbb{R}I_n$.  Under the standard representation $v \mapsto vg$, the group
$G$ has three orbits on
$\mathbb{R}^n$: the light cone $\mathcal{N}$, the time-like region $\{v :
v\Phi v^\top < 0\}$, and the space-like region $\{ v : v \Phi v^{\top} > 0\}$.

For any $A \in \mathfrak{gl}(n)$, define the $\Phi$-trace of $A$ as
$\tr_\Phi(A) = \sum_{ij} \Phi_{ij}A_{ji} = \tr(\Phi A)$.  Let $\mathcal{S}$
denote the vector space of $\Phi$-traceless symmetric matrices, so
\begin{equation} \mathcal{S}= \left\{ A \in \mathrm{Sym}^2(V) ~:~
\tr_\Phi(A)=0\right\}.\end{equation} Note that $\ver_2(\mathcal{N}) =
\mathcal{S} \cap \ver_2(V)$ is the space of rank-one, symmetric,
$\Phi$-traceless matrices.
%

Consider the ``orthogonal adjoint'' representation of $O(n,\Phi)$ on
the vector space $\mathfrak{gl}(n)$ given as
\begin{equation}
Od_g(X) = g^{\top} X g,\ X \in \mathfrak{gl}(n)
\label{Odaction}
\end{equation}
Let $Od(G)$ denote the subgroup of $Od(GL(n))$ defined as $Od(CO(n,\Phi))$ in
Appendix~\ref{neg}, so 
\begin{equation}
Od(G) = \{ A \mapsto \lambda g^{\top} A g,\ g \in O(n,\Phi), \lambda \in
\mathbb{R}^\times\}.
\label{OdG}
\end{equation}
The Lie algebra of this group yields a faithful representation of $\mathfrak{g}$,
namely
\begin{equation}
od(\mathfrak{g}) = \{ A \mapsto X^{\top}A + AX,\ X \in \mathfrak{g}\}.
\label{odaction}
\end{equation}

\begin{lemma}
The group $Od(G)$ preserves $\mathcal{S}$, $\ver_2(V)$, and
$\ver_2(\mathcal{N})$ as varieties in $V \odot V$.
Moreover, $\mathcal{S}$ is an irreducible $Od(G)$-module.
\end{lemma}
\begin{proof}
For any $A \in \mathcal{S}$, then $\sum_{ij}\Phi_{ij}Od_g(A)_{ij} =
\sum_{ijkl}\Phi_{ij}g_{ki}A_{kl}g_{lj} = \sum_{kl} \left( \sum_{ij}
g_{lj}\Phi_{ji}g_{ki} \right) A_{kl} = \lambda \sum_{kl} \Phi_{kl}A_{kl} = 0$.
For any $A \in \ver_2(V)$, there exists $v \in V$ such
that $A=v^{\top}v$. Therefore $Od_g(A) = g^{\top}v^{\top}vg = (vg)^{\top} (vg) \in
\ver_2(V)$.
Since $\ver_2(\mathcal{N})$ equals $\ver_2(V) \cap \mathcal{S}$, it is
also preserved.  The action $Od_g$ is transitive on the non-zero elements of
$\ver_2(V)$
because $\ver_2$ is a bijection onto its image and $G$ is transitive on
$\mathbb{P}V$.

Over $\mathbb{C}$, $\mathcal{N}$ is a non-degenerate affine variety. 
Since $\ver_2(\mathcal{N})$ is in bi-rational correspondence with
$\mathcal{N}$, the action of $Od(G)$ on the non-zero elements of $\ver_2(\mathcal{N})$ is transitive
and irreducible.  In the language of Cartan \cite{Cartan1981}, $\ver_2(\mathcal{N})$ is an
irreducible Euclidean representation for $Od(G)$ spanning the vector space
$\mathcal{S}$.  In particular, this implies that $\mathcal{S}$ is an
irreducible $Od(G)$-module.  
\end{proof}

As an $Od(G)$-module, $V \odot V$ decomposes into
irreducible submodules as $\mathcal{S} +
\mathbb{R}\Phi^{-1}$.  The projection onto the second component is $A \mapsto
\frac1n\tr(\Phi A)\Phi^{-1}$.

Of course, $\mathfrak{gl}(n)$ also admits another action of 
$O(n,\Phi)$, the
well-known (right) adjoint representation.
\begin{equation}
Ad_g(A) = g^{-1}A g.
\label{Adaction}
\end{equation}
Augmenting this action with scalings as in Appendix~\ref{neg},
consider the group $Ad(G) = \{ A \mapsto \lambda g^{-1} A g,\ g \in O(n,\Phi),
\lambda\neq 0\}$.
Let $\mathcal{R}= \{ a: a^{\top} = \Phi^{-1}a\Phi,  \tr a=0\}$, which is the 
space of traceless Ricci tensors for the nondegenerate
symmetric bilinear form $\Phi$.  Then
$\mathfrak{gl}(n)$ decomposes into the $Ad(G)$-submodules $\mathbb{R}I +
\mathfrak{so}(n,\Phi) + \mathcal{R}$.
The first two summands together are $\mathfrak{g}$, the Lie algebra of $G$. 
The projections are given by 
\begin{equation}
\begin{split}
\pi_{\mathbb{R}I}: a &\mapsto \frac1{n} \tr(a) I,\\
\pi_{\mathfrak{so}(n,\Phi)}: a &\mapsto \frac12\left(a - \Phi a^{\top}
\Phi^{-1}\right),\\
\pi_{\mathcal{R}}: a &\mapsto \frac12\left(a + \Phi a^{\top} \Phi^{-1}\right) -
\frac1{n} \tr(a) I,\\
\pi_{\mathfrak{g}}: a  &\mapsto \pi_{\mathbb{R}I}(a)+\pi_{\mathfrak{so}(n,\Phi)}(a).
\end{split}
\end{equation}
Moreover, the $Ad(G)$-module $\mathcal{R}$ is isomorphic to the
(irreducible) $Od(G)$-module $\mathcal{S}$, as seen in Figure~\ref{reps}.  
Generally, the $Ad(G)$-module $\mathfrak{so}(n,\Phi)$ is not irreducible.
The corresponding decomposition for $Od$ is $\mathfrak{gl}(n) =
\mathbb{R}\Phi^{-1} + \mathfrak{so}(n) + \mathcal{S}$ with the obvious
projections. 
\begin{lemma}
The bilinear pairing $\pair{\cdot,\cdot} : \mathcal{S}^*
\otimes \mathcal{S}^* \to \mathbb{R}$ defined by $\pair{A,B} = \frac1n\tr(\Phi A \Phi
B)$ is non-singular and is $Od(G)$-invariant, up to scale.  Therefore, $Od(G)$
is a subgroup of the conformal group $CO(m, \pair{\cdot,\cdot})$.
\end{lemma}
\begin{proof}
Trace is cyclic, so the pairing is symmetric.  Because $\Phi$ is non-singular,
$\pair{\cdot,\cdot}$ is also non-singular.  To check invariance up to scale, it suffices
to check over $Od(O(n,\Phi))$.  For any $g \in Od(n,\Phi)$,
\[ 
\pair{ Od_g(A), Od_g(B)} = \tr(\Phi g^{\top} A g \Phi g^{\top} B g ) = \tr( g \Phi g^{\top} A
g \Phi g^{\top} B ) = \pair{A,B}. \]
\end{proof}
Though the pairing is nonsingular, it does have null directions if $\Phi$ does.  In
particular, any element of $\ver(\mathcal{N})$ is null.  The trilinear form $\pair{ A^1, A^2, A^3 } = \tr( \Phi A^1 \Phi A^2 \Phi
A^3 )$ is also fully symmetric.
Under the isomorphism from $Od(G)$-modules to $Ad(G)$-modules, it is seen that 
the pairing $\pair{\cdot, \cdot }$ for
$\mathcal{S}$ corresponds to the ``trace form'' pairing $(R,S) \mapsto \tr(RS)$ for $R, S
\in \mathfrak{gl}(n)$ using the $Ad$ action.

Let $Od^\dagger$ denote the adjoint for the pairing, $\pair{Od_g^\dagger(A),B}
= \pair{A,Od_g(B)}$, so $Od_g^\dagger(A) = g^{\top,-1}Ag^{-1} = Od_{g^{-1}}(A)$.
A dual group $Od^*(G)$ acts on
$\mathcal{S}^*=\mathrm{Hom}(\mathcal{S},\mathbb{R})$ with the rule
$F(Od_{g}(A)) = Od_{g^{-1}}^*(F)(A)$, so 
\begin{equation}
\sum_{ij} F^{ij} (g^{\top}Ag)_{ij} =
\sum_{ijkl} F^{ij} g_{ki} A_{kl} g_{lj} = \sum_{kl} (gC g^{\top})_{kl} A_{kl} =
\sum_{kl}(gCg^{\top})_{kl} A_{kl}.
\end{equation}
Therefore, $\mathcal{S}^*$ embeds in the symmetric matrices as $\{
\Phi A \Phi : A \in \mathcal{S}\}$ with the action $Od_g^*$ acting like
$Od_{g^{\top,-1}}$.

Using the $Od^*$ and $Od^\dagger$ identifications of $\mathcal{S}$ with $\mathcal{S}^*$, the space $\mathrm{Hom}(\mathcal{S},\mathcal{S}) =  \mathcal{S} \otimes \mathcal{S}^*$ 
is identified with $\mathrm{Hom}(\mathcal{S},\mathcal{S}^*)=\mathcal{S}^*
\otimes \mathcal{S}^*$ and with $\mathrm{Hom}(\mathcal{S}\otimes\mathcal{S},\mathbb{R})=(\mathcal{S} \otimes
\mathcal{S})^*$ using the bilinear form $\pair{\cdot,\cdot}$ as seen here:
\begin{equation}
\left\{ A \mapsto Q(A) \right\} \leftrightarrow 
\left\{  A \mapsto \Phi Q(A) \Phi \right\}  \leftrightarrow 
\left\{ A{\otimes}B \mapsto \pair{Q(A),B} \right\}.
\end{equation}
These have some important submodules that arise frequently here
\begin{equation}
\begin{split}
\mathcal{Q}_{0}  &= \{ A \mapsto Q(A) = c A,\ c \in \mathbb{R}\} \\
&
\leftrightarrow \{ A \otimes B \mapsto c \pair{A,B}, c \in \mathbb{R} \} \cong
\mathbb{R}\\ 
\mathcal{Q}_{1}  &= \{ A \mapsto Q(A) = \textstyle{\frac12}\left( A \Phi C + C \Phi
A\right) - \pair{A,B}\Phi^{-1} ,\  C \in \mathcal{S} \}\\
& \leftrightarrow \{ A\otimes B \mapsto \pair{A,B,C},\ C \in
 \mathcal{S}\} \cong \mathcal{S}\\
\mathcal{Q}_{-}  &= \{ A \mapsto Q(A), \pair{Q(A),B} = - \pair{A,Q(B)} \}
\cong \mathcal{S}^* \wedge \mathcal{S}^*,\\
\mathcal{Q}_{+}  &= \{ A \mapsto Q(A), \pair{Q(A),B} = \pair{A,Q(B)} \}  \cong
\mathcal{S}^* \odot \mathcal{S}^*
\end{split}
\end{equation}
Of course, $\mathcal{Q}_{-}$ is the set of endomorphisms of $\mathcal{S}$ that are
anti-self-adjoint for $\pair{\cdot,\cdot}$, and $\mathcal{Q}_{+}$ is the set of
endomorphisms of $\mathcal{S}$ that are self-adjoint for $\pair{\cdot,\cdot}$.
When considered as elements of $\mathcal{S}^* \otimes \mathcal{S}^*$ using the
pairing, they describe $\mathcal{S}^* \wedge \mathcal{S}^*$ and $\mathcal{S}^*
\odot \mathcal{S}^*$, respectively.
Both $\mathcal{Q}_0$ and $\mathcal{Q}_1$ are submodules of $\mathcal{Q}_{+}$,
and $\mathcal{Q}_0 + \mathcal{Q}_1$ is identified with the space of symmetric
matrices by mapping the parameters $c$ and $C$ to $c\Phi + C \in
\mathbb{R}\Phi^{-1} + \mathcal{S}$.
There are preferred projections onto these two components, given again by the
identification of $\mathcal{S}$ and $\mathcal{S}^*$,
\begin{equation}
\begin{split}
\Pi_0&: A^* \odot B^* \mapsto \pair{A,B}\Phi^{-1}\\ 
\Pi_1&: A^* \odot B^* \mapsto  \frac12\left( A \Phi B + (A \Phi B)^{\top}
\right)  -
\pair{A,B}\Phi^{-1} \in \mathcal{S}.
\end{split}
\end{equation}
Writing $\Pi= (\Pi_0 + \Pi_1):\mathcal{S}^* \odot \mathcal{S}^* \to
(\mathbb{R}\Phi^{-1} + \mathcal{S})$, set
$\mathcal{Q}_2 = \ker \Pi \subset \mathcal{Q}_{+}$, so $\mathcal{Q}_2
\cong (\mathcal{S}^* \odot \mathcal{S}^*)/(\mathbb{R}\Phi^{-1} + \mathcal{S})$
and $\mathcal{Q}_{+} = \mathcal{Q}_{0} + \mathcal{Q}_{1} + \mathcal{Q}_{2}$.
Overall, $\mathcal{S}^* \otimes \mathcal{S}^* = \mathcal{Q}_{-} +
\mathcal{Q}_{0} + \mathcal{Q}_{1} + \mathcal{Q}_{2}$.

\begin{figure}[h]
\[\xymatrix{
& & g^{\top} A g \ar[lldd] & & \\
& & \mathcal{S}\ar[dl]_{\Phi \bullet} & &\\
g^{-1}Rg \ar[rrdd] & \mathcal{R}\ar[dr]_{\bullet\Phi} & &
\mathcal{R}^*\ar[ul]_{\bullet \Phi^{-1}} &  g^{\top} R^* g^{\top,-1} \ar[lluu]\\
& & \mathcal{S}^*\ar[ur]_{\Phi^{-1}\bullet} & &\\
& & g^{-1}A^*g^{\top,-1} \ar[rruu] & & }\]
\caption{The ($\Phi$ or $I$)-traceless and ($I$ or $\Phi$)-symmetric
$G$-modules and their duals.}
\label{reps}
\end{figure}
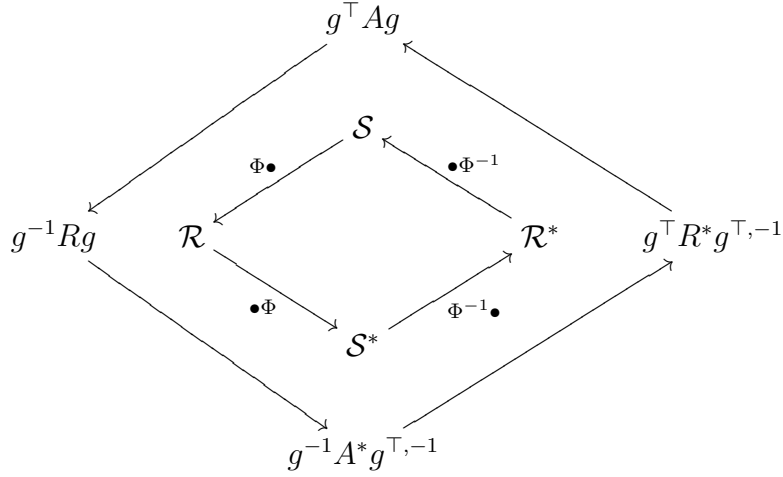

\section{$Od(G)$-Structures}
Let $M$ denote a smooth manifold of dimension $m=\frac12n(n+1)-1$, and let
$\mathcal{F}(M)$ denote the $\mathcal{S}$-valued co-frame bundle over $M$, meaning that
elements of $\mathcal{F}_p(M)$ are linear isomorphisms $\eta_p:\mathbf{T}_pM \to
\mathcal{S}$.  Conventionally $\mathcal{F}(M)$ is a principal right
$GL(\mathcal{S})$-bundle.


\begin{defn}\footnote{In an earlier preprint, these were called ``Veronese structures,'' but in
retrospect, that name is misleading in this context and is better used for a
different but related structure that will appear elsewhere.}
For $M$ of dimension $m=\frac12n(n+1)-1$, an ``$Od(G)$-structure'' with
respect to $\Phi$ is a reduction of the $\mathcal{S}$-valued co-frame 
bundle by the action of $Od(G)$.  In particular, an $Od(G)$-structure is a
principal right $G$-bundle $\mathcal{B}\to M$.
\label{maindef}
\end{defn}

Let $\omega$ denote the tautological $\mathcal{S}$-valued one-form of
an $Od(G)$-structure $\mathcal{B} \to M$, so
$\omega_b(X) = b \circ \pi (X)$ for all $X \in \mathbf{T}_b \mathcal{B}$.  Let
$\eta$ be a local section, $\eta:M \to \mathcal{B}$.  Define a local
trivialization $H:M \times G \to \mathcal{B}$ by 
$H(p,g) = Od_g(\eta_p)$.  Then $\eta^*(\omega) = \eta$ and $H^*(\omega) =
Od_g \circ \eta$.  Then
\begin{equation}
\begin{split}
H^*(\mathrm{d}\omega) &= \mathrm{d}(H^*(\omega)) \\
&= \mathrm{d}(g^{\top}\eta g) = \mathrm{d}g (g^{\top})^{-1} g^{\top} \eta g +
Od_g\circ \mathrm{d}\eta - g^{\top} \eta g g^{-1} \mathrm{d}g\\
&= (g\mathrm{d}g)^{\top} \wedge H^*(\omega) - H^*(\omega)\wedge(g^{-1}\mathrm{d}g)
+Od_g\circ \mathrm{d}\eta \\
\end{split}
\end{equation}
Notice that $\theta=(H^*)^{-1}(g^{-1}\mathrm{d}g)$ is a $\mathfrak{g}$-valued
pseudo-connection on $\mathcal{B}$ with apparent torsion
$T(\omega\wedge\omega)=(H^*)^{-1}(Od_g\circ\mathrm{d}\eta)$.  The apparent
torsion $T$ is a function on $\mathcal{B}$ valued in $\mathcal{S}\otimes
(\mathcal{S}^*\wedge \mathcal{S}^*)$. 
Cartan's first structure equation for $Od(G)$-structures is thus 
\begin{equation}
\mathrm{d}\omega_{ij} = \sum_{k} \left( \theta_{ki} \wedge \omega_{kj} - \omega_{ik}
\wedge\theta_{kj}\right) + T_{ij}(\omega\wedge\omega).
\label{Cartan1}
\end{equation}

\begin{defn}
An $Od(G)$-structure $\mathcal{B}\to M$ is said to be ``embeddable'' if there exists a
bundle embedding into $CSp(N)^o \to \Lambda^o$.  An $Od(G)$-structure is said
to be ``locally embeddable near $b$'' if there is an open neighborhood $U$ of
$b\in \mathcal{B}$ such that $U$ is embeddable.
\end{defn}

The discussion in Section~\ref{back} and the fundamental lemma of Lie groups immediately provide a simple characterization of
locally embeddable structures.
\begin{lemma}
The following are equivalent for an $Od(G)$-structure $\mathcal{B}\to M$ with $p\in M$.
\begin{enumerate}
\item $\mathcal{B}$ is locally embeddable near $b$ for some $b \in \mathcal{B}_p$;
\item there exists a local inclusion $i:M \to \Lambda^o$ near $p$ such that $i(M)=F^{-1}(0)$
for a Hessian PDE $F=0$;
\item there is an $\mathfrak{sp}(n)$-valued $1$-form $\mu$ defined in a
neighborhood of $b \in \mathcal{B}_p$ such that 
\begin{equation}
\mu=\begin{pmatrix}
\beta & \gamma \\ \alpha & -\beta^{\top}
\end{pmatrix}
\end{equation}
with $\alpha=\alpha^{\top}$, $\gamma=\gamma^{\top}$, $\mathrm{d}\mu +
\mu\wedge\mu=0$, and such that $\alpha$ is semi-basic and of maximum rank $m$
on $\mathcal{B}$.
\end{enumerate}
\label{embeddable}
\end{lemma}

Henceforth, only (locally) embeddable structures are considered.

\subsection{Embeddable Torsion and The First Fundamental Lemma}
In this section, we apply the first step of Cartan's method of equivalence to
normalize the first-order structure equations and find global forms of the connection
and torsion for embeddable $Od(G)$-structures \cite{Ivey2003}.

Before beginning the method, note an important algebraic curiosity that is
inherent to these structures.  For an arbitrary Lie group $H \subset
GL(\mathcal{S})$, consider a $H$-structure $\mathcal{B}$ over a manifold $M$ with tangent space
$\mathbf{T}_pM \cong \mathbb{R}^m =\mathcal{S}$.  In Cartan's tradition, the local
equivalence of $H$ structures can be understood by determining how much apparent torsion
$T:\mathcal{B} \to \mathcal{S}\otimes(\mathcal{S}^* \otimes \mathcal{S}^*)$ can be absorbed by making an alteration of
the $\mathfrak{h}$-valued pseudo-connection of the form $\theta \mapsto \theta
+ P(\omega)$ for $P:\mathcal{B} \to \mathfrak{h}\otimes \mathcal{S}^*$.
Thus, the solution to the equivalence problem of $H$-structures
involves the computation of the skewing map $\delta$ that defines the exact
sequence 
\begin{equation}
0 \to \mathfrak{h}^{(1)} \to \mathfrak{h}\otimes \mathcal{S}^*
\overset{\delta}{\to} \mathcal{S}\otimes(\mathcal{S}^* \wedge
\mathcal{S}^*) \to H^{0,2}(\mathfrak{h}) \to 0
\label{sequence1}
\end{equation}
Informally, $H^{0,2}(\mathfrak{h})$ is the space where invariant torsion is
valued, and $\mathfrak{h}^{(1)}$ controls the uniqueness of global connections
with a given essential torsion.
In the most historically important equivalence problems, the action by the Lie
group $H$ on the co-frames of the manifold $M^m$ is defined by the
\emph{standard representation} of $H$ as embedded in $GL(\mathcal{S})$, so the inclusion
of $\mathfrak{h}$ into $\mathcal{S}\otimes \mathcal{S}^*$ is the identity embedding, and the map
$\delta$ is given by the composition of the maps
\[\mathfrak{h}\otimes \mathcal{S}^* \overset{(i,1)}{\to} (\mathcal{S} \otimes
\mathcal{S}^*) \otimes \mathcal{S}^*
\overset{(1,\wedge)}{\to} \mathcal{S} \otimes(\mathcal{S}^* \wedge
\mathcal{S}^*).\] However, for $Od(G)$-structures, the action of
$\mathfrak{g}$ on $\mathcal{S}$ is given by a different
representation, namely Equation~(\ref{odaction}).  
Thus, for any $P \in \mathfrak{g} \otimes \mathcal{S}^*$, the image of $P$ in
$\mathcal{S} \otimes \mathcal{S}^* \otimes \mathcal{S}^*$ is the map
$(A,B) \mapsto P(A)^{\top} B + B P(A)$, so
$\delta(P)(A,B) = \frac12 
\left( P(A)^{\top} B + B P(A) - P(B)^{\top}A - AP(B)\right)$.  As an
identification of two-forms valued in $\mathcal{S}$, this is written as $\delta(P)(\omega\wedge\omega) =
P(\omega)^{\top} \wedge \omega - \omega \wedge P(\omega)$.
Note the apparent sign change, which is really just a consequence of the rule
$(\alpha\wedge\beta)^{\top} = (-1)^{pq} \beta^{\top}\wedge\alpha^{\top}$ for
matrix-valued $p$- and $q$-forms $\alpha$ and $\beta$.
Conceptually, the computation of the
kernel and co-kernel of $\delta$ is still the appropriate approach, but
the computation relies on this matrix arithmetic.

We now proceed to study $\delta$ by first examining a related linear map
$\bar\Delta$ on a larger domain.
Consider the space 
\[ \mathfrak{gl}(N) \otimes \mathcal{S}^* = \left\{ \left(Y_{ij}^{kl}\right) : Y_{ij}^{kl}=Y_{ij}^{lk},
\sum_{kl}Y_{ij}^{kl}(\Phi^{-1})_{kl} = 0\right\}.\]
Let $f$ denote the map $\mathfrak{gl}(n) \otimes \mathcal{S}^* \to
\mathfrak{gl}(n) \otimes (\mathcal{S}^*\otimes \mathcal{S}^*)$ defined by
$f(Y)(A,B) = Y(A)^{\top}B$ for any $Y \in \mathfrak{gl}(n)\otimes
\mathcal{S}^*$.  Let $\bar\Delta$ denote the skew of $f$, so map $\bar\Delta(Y)(A,B) =
f(Y)(A,B) - f(Y)(B,A) = Y(A)^{\top}B - Y(B)^{\top}A$.

\begin{lemma}
The map $\bar\Delta$ is injective.
\label{kerDelta}
\end{lemma}
\begin{proof}
To study $\bar\Delta$, it is expedient to evaluate the map on 1-forms, which is
anyway the situation that is always needed for local geometry.  Fix arbitrary $Y \in \ker
\bar\Delta$, and let $\psi = Y(\omega)$.  Aside from symmetry, the 1-forms $\omega_{ij}=\omega_{ji}$
have only one linear relation, namely that $\tr(\Phi \omega)=0$.  Because
$\Phi$ is assumed to be non-degenerate, the entries of each row (or column)
of $\omega$ are independent 1-forms.

Then $0=\sum_k\psi_{ik}\wedge\omega_{kj} =  \sum_k \psi_{ik}\wedge\omega_{jk}$ for all $i,j$.  
In the case $j=1$, the 1-forms $\omega_{11}, \ldots, \omega_{1n}$ are
independent, so Cartan's lemma implies that $\psi_{ik} =
\sum_{l}C'_{ikl} \omega_{1l}$ for some functions $C'_{ikl} = C'_{ilk}$.
The case $j=2$ similarly implies that $\psi_{ik} = C''_{ilk}
\omega_{2l}$ for some functions $C''_{ilk} = C''_{ikl}$.
The case $j=n$ similarly implies that $\psi_{ik} = \sum_{l}
C'''_{ikl} \omega_{nl}$ for some functions $C'''_{ikl} = C'''_{ilk}$.
Comparing cases $j=1$ and $j=2$, it must be that $\psi_{ik} \equiv 0$ modulo
$\omega_{12}$.  Comparing cases $j=1$ and $j=n$, it must be that $\psi_{ik}
\equiv 0$ modulo $\omega_{1n}$.  So, $\psi_{ik}=0$, and $Y=0$.
\end{proof}

Let $\bar \delta$ denote the map $\mathfrak{gl}(N)\otimes \mathcal{S}^* \to
(V \odot V) \otimes (\mathcal{S}^* \wedge \mathcal{S}^*)$
that is defined by 
\begin{equation}
\begin{split}
\bar{\delta}(Y)(A,B) &= \frac12\left(od_{Y(A)}(B)  - od_{Y(B)}(A)\right)\\
&=
\frac12\left( Y(A)^{\top}B + (Y(A)^{\top}B)^{\top} - 
Y(B)^{\top}A - (Y(B)^{\top}A)^{\top}\right)\\
&= \bar\Delta(Y)(A,B) + (\bar\Delta(Y)(A,B))^{\top}.
\end{split}\end{equation}
If $Y$ happens to be a change of connection valued in
$\mathfrak{g}\otimes \mathcal{S}^*$, then $\bar\delta(Y) \in \mathcal{S}
\otimes (\mathcal{S}^*\wedge \mathcal{S}^*)$
is the resulting change of torsion.  Thus the skewing map is computed as
$\delta = \bar\delta|_{\mathfrak{g}\otimes \mathcal{S}^*}$.  
\begin{lemma}
\[ \ker \bar \delta = \left\{ A \mapsto c\Phi A, c \in \mathbb{R} \right\} 
+ \left\{ A \mapsto \Phi C \Phi A, C \in \mathcal{S}\right\}.
\] 
Moreover, $\mathfrak{g}^{(1)}=\ker \delta = \ker \bar\delta \cap (\mathfrak{g}\otimes
\mathcal{S}^*) = 0$.  Therefore, for any choice of co-kernel of $\delta$,
every $Od(G)$-structure admits a unique and global connection such that the torsion map $T$
takes values in that co-kernel.
\label{kerdeltabar}
\end{lemma}
\begin{proof}
Suppose $Y \in \ker \bar\delta \subset \mathfrak{gl}(n) \otimes \mathcal{S}^*$.  Let
$\tau_{ij} = Y_{ij}^{(kl)}\omega_{(kl)}$.
Then 
\begin{equation}
0 = \sum_a \tau_{ai}\wedge\omega_{aj} - \omega_{ia}\wedge\tau_{aj},\ \forall
i,j.
\end{equation}
In the case $i{=}j$, this implies $0=\sum_a \tau_{ai}\wedge\omega_{ai}$.  For
fixed $i$, the collection $\{\omega_{1i},\ldots,\omega_{ni}\}$ is linearly independent, so the
Cartan Lemma implies that there exist functions
$C(i)^{a}_{b}=C(i)^{b}_{a}$ such that $\tau_{ia} = \sum_b
C(i)_{a}^{b}\omega_{ib}$ for all $a,i$.  
Then, for any $i{\neq}j$, 
\begin{equation}
\begin{split}
0&=\sum_a\left(\tau_{ai}\wedge\omega_{aj}+\tau_{aj}\wedge\omega_{ai}\right)\\
&=\sum_{a,b}\left(C(i)_a^b\omega_{bi}\wedge\omega_{aj}+C(j)_a^b\omega_{bj}\wedge\omega_{ai}\right)\\
&= \sum_a \left(C(i)_a^a\omega_{ai}\wedge\omega_{aj}+C(j)_a^a\omega_{aj}\wedge\omega_{ai}\right)\\
&\phantom{=}
+\sum_{a<b}\left(\left(C(i)_a^b-C(j)_b^a\right)\omega_{bi}\wedge\omega_{aj}+
\left(C(j)_a^b-C(i)_b^a\right)\omega_{bj}\wedge\omega_{ai}\right).
\end{split}
\end{equation}
Each term in the previous expression is linearly independent, so $C(i)^a_b =
C(j)^a_b = C^a_b = C^b_a$ for all $i,j,a,b$.  Thus, the kernel of $\bar\delta$
is isomorphic to $V \odot V$ as determined by these
$\frac12n(n+1)$ constants.
It is now easy to check that $\ker\bar\delta$ intersects trivially with $\mathfrak{g}\otimes \mathcal{S}^*$.
\end{proof}
Note that the output of $\bar\delta(P)$ is a symmetric matrix, but it is not
$\Phi$-traceless for general $P \in \mathfrak{gl}(n)\otimes \mathcal{S}^*$, so
general $\bar\delta(P)$ cannot represent an apparent torsion in $\mathcal{S} \otimes (\mathcal{S}^*\wedge \mathcal{S}^*)$. 
Consider the subspace 
\begin{equation}
\mathcal{P} = \bar\delta^{-1}\left(\mathcal{S}\otimes(\mathcal{S}^*\wedge
\mathcal{S}^*)\right) = 
\{ P \in \mathfrak{gl}(n)\otimes \mathcal{S}^* : (\tr_\Phi\otimes 1)(\bar\delta P)=0\},
\end{equation}
along with its image, $\mathcal{E} = \bar\delta(\mathcal{P}) \subset
\mathcal{S} \otimes (\mathcal{S}^* \wedge
\mathcal{S}^*)$.  Of course, $\mathfrak{g}\otimes \mathcal{S}^* \subset \mathcal{P}$, and
$\delta(\mathfrak{g}\otimes \mathcal{S}^*) \subset \mathcal{E}$.  As justified by 
Lemma~\ref{embeddabletorsion}, $\mathcal{E}$ is called the space of ``embeddable torsion.''

\begin{lemma}
If $\mathcal{B} \to M$ is an embeddable $Od(G)$-structure, then for any
(local) pseudo-connection $\theta$, the associated apparent torsion is a map $T:\mathcal{B} \to \mathcal{E}$.  In
particular,
$\mathcal{B}$ admits a function $P:\mathcal{B} \to \mathcal{P}$, unique up to
$\ker \bar{\delta}$, such that  
\[ \mathrm{d}\omega = (\theta + P(\omega))^{\top}\wedge\omega -
\omega\wedge(\theta + P(\omega)).\]
\label{embeddabletorsion}
\end{lemma}
\begin{proof}
Fix a pseudo-connection $\theta$ on $\mathcal{B}$ with apparent torsion
$T(\omega\wedge\omega)$.  Let $\alpha$, $\beta$, and $\gamma$ denote the
blocks of the Maurer--Cartan form of $Sp(n)$, as in Equation~\ref{mc}.  

If $\mathcal{B}$ is locally embeddable via a bundle embedding $h$, then the
$\frac12n(n+1)$ semi-basic one-forms $h^*(\alpha_{jk})$, $j \leq k$ must 
have a single linear relation, namely that $\tr(\Phi h^*(\alpha))=0$.  Thus,
the components of the $\mathcal{S}$-valued tautological form $\omega$ on
$\mathcal{B}$ are given by $\omega_{jk} = h^*(\alpha_{jk})$.
The $\mathfrak{gl}(n)$-valued one-form $h^*(\beta)$ may have both vertical and
semi-basic components, so write $h^*(\beta_{kl})= \theta' +
\tau$ where $\tau \equiv 0 \mod \{\omega_{kl}\}$ and $\theta' \equiv  0 \mod
\{ \theta_{kl}\}$.
Then
\begin{equation}
\begin{split}
0&= \mathrm{d}\omega - h^*(\mathrm{d}\alpha) \\
&= (\theta - h^*(\beta))^{\top} \wedge \omega - \omega \wedge (\theta -
h^*(\beta)) + T(\omega\wedge\omega) \\
&=
\begin{cases}
(\theta - \theta')^{\top}\wedge \omega - \omega \wedge (\theta - \theta') \\
T(\omega\wedge\omega) - \left( \tau^{\top} \wedge \omega - \omega \wedge
\tau\right).
\end{cases}
\end{split}
\end{equation}
By Lemma~\ref{injective}, $\theta'=\theta$.
Also, since $\tau$ is semi-basic, one may write $\tau = P(\omega)$ for some $P:
\mathcal{B} \to \mathfrak{gl}(N) \otimes \mathcal{S}^*$.  Then 
$T=\bar{\delta}(P)$.  Because $T$ is \emph{a priori} valued in $\mathcal{S} \otimes
(\mathcal{S}^*\wedge \mathcal{S}^*)$, it must be that $P \in \mathcal{P}$ and
$T \in \mathcal{E}$.
\end{proof}

The next task is to write $\theta$ and $T$ in a preferred way so that the
structure equations of an embeddable $Od(G)$-structure are global. Consider
Figure~\ref{diagram}, with $\mathcal{T} = \mathcal{E}/\delta  \subset
H^{0,2}(\mathfrak{g})$.  Since $\ker \delta$ and $\mathfrak{g}\otimes
\mathcal{S}^*$ are naturally subspaces of $\mathcal{P}$, a preferred
representative of $[T]$ can given by specifying any section $\sigma: \mathcal{T} \to
\mathcal{P}$, which then yields a corresponding decomposition $\mathcal{P} = \ker \bar\delta +
(\mathfrak{g}\otimes \mathcal{S}^*) + \sigma(\mathcal{T})$.

The next few lemmas specify a preferred cokernel $\mathcal{T}'
=\sigma(\mathcal{T})$ by analyzing the
sub-modules of $\mathfrak{gl}(n)\otimes \mathcal{S}^*$ under 
the $Ad(G) \otimes Od^*(G)$ action.

\begin{figure}[h]
\[\xymatrix{
&&&0\ar[d]\\
&&&\mathfrak{g}\otimes{\mathcal{S}}^*\ar[d]_{\delta}\\
0\ar[r]&{\ker\bar\delta}\ar[r]&
\mathcal{P}\ar[r]^{\bar\delta}&\mathcal{E}\ar[d]\ar[r]&0
&&B\ar@/^{2pc}/[lll]_{T}\ar@/^{2pc}/[llld]^{[T]}\ar@/_{2pc}/[lllu]^{\pi_\mathfrak{g}\circ P}\ar@/_{9pc}/[llll]_{P}\\
&&&\mathcal{T}\ar@{.>}[ul]_{\sigma}\ar[d]\\
&&&0
}\]
\caption{The torsions of an embeddable $Od(G)$-structure.  The quantity $P$
depends on the choice of the section $\sigma$.}
\label{diagram}
\end{figure}
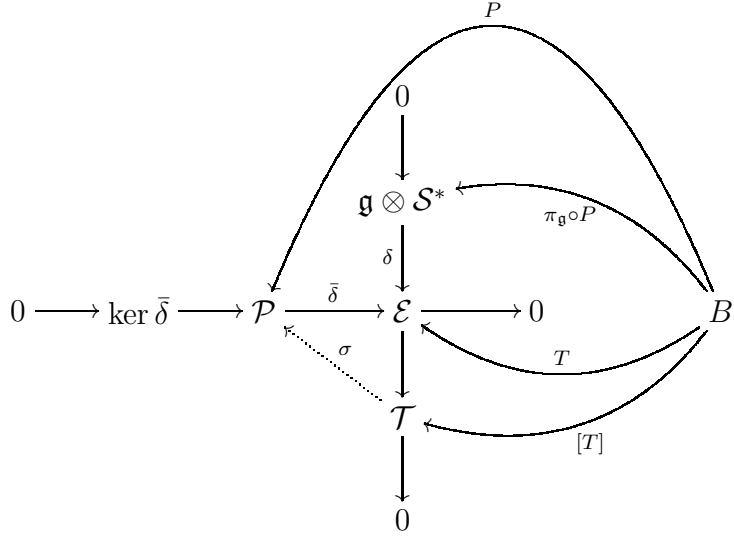

\begin{lemma}
Let $\mathcal{P}_\mathcal{R} = \{ P \in \mathcal{P} : \pi_\mathfrak{g}(P(A)) =
0\ \text{for all $A \in \mathcal{S}$}\} = (\pi_\mathcal{R} \otimes
1)(\mathcal{P})$
Then $\mathcal{P}_\mathcal{R}$ is isomorphic to $\mathcal{Q}_{+}$, so
$\mathcal{P} = \mathcal{Q}_+ \oplus (\mathfrak{g}\otimes
\mathcal{S}^*)$.
\label{PR}
\end{lemma}
\begin{proof}
It suffices to define a map $e$ such that the sequence 
\begin{equation}
0 \to \mathfrak{g}\otimes \mathcal{S}^* \to \mathcal{P} \overset{e}{\to}
\mathcal{Q}_{+} \to 0
\label{Eseq}
\end{equation}
is exact. 
For any $P:\mathcal{S} \to \mathfrak{gl}(n)$, define the map
$e(P):\mathcal{S} \to \mathcal{S}$ by
$e(P)(A)= \Phi^{-1}\pi_\mathcal{R}(P(A))$. 
Let $E(P)\in \mathcal{S}^* \otimes \mathcal{S}^*$ denote the bilinear pairing
associated to $e(P)$ using the identifications from Section~\ref{infingeom}. So,
\begin{equation}
\begin{split}
E(P)(A,B) &= \pair{e(P)(A),B}\\
&= 
\pair{ \Phi^{-1}\frac12 \left(P(A) + \Phi P(A)^{\top}\Phi^{-1}\right) -
\frac1n \tr(P(A))\Phi^{-1}, B} \\
&=
\frac1{2n}\left(\tr(P(A)\Phi B) + \tr(B \Phi P(A)^{\top})\right) \\
&= \frac1n\tr(P(A)\Phi B).
\end{split}
\label{evaluation}
\end{equation}
Note that $\tr_\Phi(\bar{\delta}(P)(A,B)) = \tr(P(A)\Phi B) -
\tr(P(B)\Phi A)$, so $P \in \mathcal{P}$ if and only if $E(P) \in
\mathcal{S}^* \odot \mathcal{S}^*$, which is true if and only if $e(P) \in
\mathcal{Q}_{+}$.  Moreover, $e(P)=0$ for all $P \in \mathfrak{g}\otimes
\mathcal{S}^*$. 

To prove that $\mathfrak{g}\otimes \mathcal{S}^*$ is the entire kernel of $e$,
it suffices to prove that dimension of $\mathcal{P}$ is $\frac12 m(m+1) + \dim
(\mathfrak{g}\otimes \mathcal{S}^*)$ or equivalently that the co-dimension of $\mathcal{P}$ in
$\mathfrak{gl}(n) \otimes \mathcal{S}^*$ is $\frac12m(m-1)$.
To do this, it suffices to prove that  the map $(\tr_\Phi \otimes 1) \circ
\bar{\delta}: \mathfrak{gl}(n)\otimes \mathcal{S}^* \to \mathbb{R}\otimes
(\mathcal{S}^*\wedge \mathcal{S}^*)$ is a surjection.   

Let $\hat{P}(A) = P(A)\Phi$; this changes the action on the image from $Ad$ to
$Od^*$, which effectively reveals the isomorphism
between $\mathcal{Q}_-$ and $\mathcal{S}^* \wedge \mathcal{S}^*$.
 Any $\hat{P} \in \mathfrak{gl}(n) \otimes \mathcal{S}^*$ may be written as 
$\hat{P}(A)_{ij} = \frac12 \sum_{ab} C_{ij}^{ab} A_{ab}$ such that $C_{ij}^{ab} =
C_{ij}^{ba}$.  Then  
\begin{equation}
\begin{split}
\tr(P(A) \Phi B - P(B) \Phi A) &= \tr(\hat{P}(A)B - \hat{P}(B)A) \\
&= \sum_{ij}\hat{P}(A)_{ij}B_{ij} - \hat{P}(B)_{ij}A_{ij}\\
&= \frac12 \sum_{ijab} C_{ij}^{ab}A_{ab}B_{ij} -
C_{ij}^{ab}B_{ab}A_{ij}\\
&=  \sum_{ijab} C_{ij}^{ab}(A_{ij}\wedge B_{ab}).
\end{split}
\end{equation}
Hence, any element of $\mathcal{S}^* \wedge \mathcal{S}^*$ may be obtained by
choosing appropriate $C_{ij}^{ab}$.  
\end{proof}

\begin{lemma}
The space of essential torsion, $\mathcal{T}$, is isomorphic to
$\mathcal{Q}_{2} = \ker \Pi$.
In particular, identifying $\mathcal{T}' = e^{-1}(\mathcal{Q}_{2})$ as a
co-kernel of $\bar\delta$ allows the
splittings $\mathcal{P}_\mathcal{R} = (\pi_\mathcal{R} \otimes 1)(\ker \bar
\delta) + \mathcal{T}'$ and $\mathcal{P} = (\mathfrak{g}\otimes \mathcal{S}^*)
+ \ker(\bar{\delta}) + \mathcal{T}'$.
\label{fullytraceless}
\end{lemma}
\begin{proof}
Consider the pair of exact sequences
\begin{equation}
\xymatrix{
0\ar[r]&{(\pi_\mathcal{R}\otimes 1)(\ker\bar\delta)}\ar[r]&\mathcal{P}_\mathcal{R}\ar[d]^{e}_{\sim}\ar[r]^{\bar\delta}&\mathcal{T}\ar[r]&0\\
0 & (\mathcal{S}\oplus \Phi^{-1}\mathbb{R})\ar[l]&  \mathcal{Q}_{+}
\ar[l]_-{\Pi} & \mathcal{Q}_{2} \ar[l] & 0. \ar[l]\\
}
\label{diag2}
\end{equation}
Recall that $e \circ (\pi_\mathcal{R} \otimes 1) = e$.
Because both sequences are exact, it suffices to
prove that the map $\Pi \circ e$ is an isomorphism from $\ker
\bar\delta$ to $\mathcal{Q}_{0} + \mathcal{Q}_{1}$.  
Suppose $K \in \ker \bar\delta$, so $K(A) = c\Phi A + \Phi C \Phi A$ for
arbitrary
$c \in \mathbb{R}$ and $C \in \mathcal{S}$.  Then 
\begin{equation}\begin{split}
E(K)(A,B) &=
\pair{ \Phi^{-1}\frac12\left( K(A) + \Phi K(A)^{\top}\Phi^{-1} \right) -
\frac1n\tr(K(A)) \Phi^{-1}, B} \\
&=\frac1{2n}\tr(c \Phi A \Phi B  + \Phi C \Phi A \Phi B +c \Phi A \Phi B  + \Phi
A \Phi C \Phi B)\\ 
&= c\pair{A,B} + \pair{A,B,C}.\\
\end{split}
\end{equation}
\end{proof}

\begin{lemma}[The First Fundamental Lemma]
An embeddable $Od(G)$-structure $\mathcal{B}$ admits a unique function
$P:B \to \mathcal{T}' \cong \mathcal{Q}_2$ and a unique connection $\theta$ such that $B$ has first structure equations:
\[ \mathrm{d}\omega = (\theta + P(\omega))^{\top}\wedge\omega - \omega \wedge (\theta
+P(\omega)). \]
Moreover, $\theta$ decomposes as $\theta = \varphi -\frac12 \lambda I$ for
$-\frac12 \lambda = \frac1n\tr(\theta)$ valued in $\mathbb{R}$ and $\varphi$ valued in
$\mathfrak{so}(n,\Phi)$.
\end{lemma}

\begin{proof}
Fix an embeddable $Od(G)$-structure $\mathcal{B}$ with (local) pseudo-connection
$\hat{\theta}$ and apparent torsion $T$ and structure equation
\begin{equation}
\mathrm{d}\omega = \theta^{\top} \wedge \omega - \omega \wedge \theta + 
T(\omega\wedge\omega).
\end{equation}
The torsion $T$ may be written as $\bar\delta(\hat{P})$ for some $\hat{P}:B
\to \mathcal{P}$ that is unique up to the addition of any $K:\mathcal{B} \to
\ker\bar\delta$.
Lemma~\ref{kerdeltabar}, implies that, among all such $\hat{P}$, there is a unique one such that
$\Pi(E(\pi_\mathcal{R} \circ \hat{P})) = 0$.

Set $\theta = \hat{\theta} + \pi_\mathfrak{g}(\hat{P}(\omega))$, and write
$P(\omega)  = \pi_\mathcal{R}(\hat{P}(\omega))$.  The decomposition of $\theta$ follows from the definition of the scaling
action in Appendix~\ref{neg}. The given structure equation now holds.
Because $\Pi(e(\hat{P}_\mathcal{R}))=0$, Diagram~(\ref{Eseq}) shows that
$\hat{P}_\mathcal{S} = e^{-1}(Y)$ for a unique $Y \in \mathcal{Q}_{2}= \ker
\Pi$. Therefore, $P \in \mathcal{T}'$. 
\end{proof}

Henceforth, the words ``connection'' and ``torsion 1-form'' refer to the
symbols $\theta=\varphi- \frac12\lambda I$ and $\tau=P(\omega)$ normalized in this way.

\subsection{Embeddable Curvature and the Structure Theorem} 
We now consider similar restrictions on the curvature,
\[
R(\omega\wedge\omega) = \mathrm{d}\theta+\theta\wedge\theta =
\mathrm{d}\varphi + \varphi \wedge\varphi - \frac12\mathrm{d}\lambda I.
\]
\emph{A priori}, the
curvature $R$ may live in $\mathfrak{g}\otimes(\mathcal{S}^* \wedge
\mathcal{S}^*)$; however, the condition that $\mathcal{B}$ is embeddable imposes
conditions determining which submodules may actually appear.
\begin{lemma}[The Second Fundamental Lemma]
Let $\mathcal{B}$ be an embeddable $Od(G)$-structure with connection
1-form $\theta$ and torsion 1-form $\tau$.  Then there exists a function $C$
valued in $(\mathbb{R}\Phi + \mathcal{S}^*) \otimes
\mathcal{S}^*$ such that 
\begin{equation}
R(\omega\wedge\omega) + \tau\wedge\tau + \nabla(\tau) = - C(\omega) \wedge \omega.
\label{str2}
\end{equation}
\end{lemma}

\begin{proof}
If $\mathcal{B}$ is embeddable, then there exists a (local) map $h : B \to
Sp(n)$ such that $h^*(\alpha) = \omega$, $h^*(\beta) = \theta + \tau$, and 
$h^*(\mathrm{d}\beta + \beta\wedge\beta + \gamma \wedge \alpha) = 0$.

Note that $h^*(\mathrm{d}\beta + \beta\wedge\beta)$ is a
$\mathfrak{gl}(n)$-valued 2-form on $\mathcal{B}$, but it decomposes into the
two terms
$R(\omega\wedge\omega) + \tau\wedge\tau$, which are semi-basic and valued in
$\mathfrak{g}$, and $\nabla(\tau)$, which is semi-basic and valued in
$\mathcal{R}$.  Thus, $h^*(\gamma) \wedge \omega$ must be semi-basic, so
$h^*(\gamma) \equiv 0 \mod \{ \omega_{jk} \}$. 

Since $\gamma=\gamma^{\top}$, $h^*(\gamma)$ must take values in
the symmetric matrices, which decompose into either $\mathbb{R}\Phi +
\mathcal{S}^*$ or $\mathbb{R}\Phi^{-1} + \mathcal{S}$ depending on the action
of $G$.  The $Ad(G)$ action on $h^*(\mathrm{d}\beta+\beta\wedge\beta)$ shows that the
appropriate action on $h^*(\gamma)$ is $Od^*(G)$. Thus $h^*(\gamma)$ may
take values in the sum of the two irreducible $Od^*(G)$-modules $\mathcal{S}^*$ and
$\mathbb{R} \Phi$.
\end{proof}

\begin{thm}[The Structure Theorem]
Suppose $\mathcal{B} \to M$ is an embeddable $Od(G)$-structure.  Then there
are unique $G$-equivariant functions 
$P:\mathcal{B} \to \mathcal{T}'$, 
$Q_{12}:\mathcal{B} \to (\mathcal{Q}_{1} + \mathcal{Q}_{2})$, 
$Q_-:\mathcal{B} \to \mathcal{Q}_{-}$, 
$r:\mathcal{B} \to \mathbb{R}$, and 
$s: \mathcal{B} \to \mathcal{S}^*$
such that $\mathcal{B}$ has structure equations
\begin{equation}
\begin{split}
\mathrm{d}\omega &= - \lambda I \wedge \omega + (\varphi +
P(\omega))^{\top}\wedge\omega - \omega \wedge (\varphi +P(\omega)),\\
&\phantom{=} \text{action by $Od$} \\
\mathrm{d}(P(\omega)) &= - \theta^{\top} \wedge P(\omega) - P(\omega) \wedge
\theta - s(\omega)\Phi\wedge\omega -
\pi_{\mathcal{R}}\left([Q_{12}(\omega) + Q_-(\omega)]\wedge\omega\right),\\
&\phantom{=} \text{action by $Ad$}\\
\mathrm{d}\varphi &= - \varphi \wedge \varphi  - P(\omega) \wedge P(\omega) +
\textstyle{\frac1n}\tr(P(\omega)\wedge P(\omega))I  - r \Phi \omega \wedge \Phi \omega  \\
&\phantom{= }
- \pi_{\mathfrak{so}(n,\Phi)}\left( [Q_{12}(\omega) + Q_{-}(\omega)]\wedge\omega\right)\\
&\phantom{=} \text{action by $Ad$}\\
\mathrm{d}\lambda &= 2 \tr(P(\omega)\wedge P(\omega)) + 2
\pi_{\mathbb{R}I}\left(Q_-(\omega) \wedge \omega\right)\\
&\phantom{=} \text{trivial action}\\
\end{split}
\label{eqnembedstr}
\end{equation}
In the general case, all of the projections are injections on the shown
representations.  In particular, $\nabla(\theta)=R(\omega\wedge\omega)$
depends only on $P$, $\nabla(P)$, and a single scalar curvature, $r$.
\label{thmembedstr}
\end{thm}

\textbf{Note!}
Although the structure equations in Theorem~\ref{thmembedstr} are true for any
non-degenerate $\Phi$, their precise formulation depends on $\Phi$.
The group $Od(G)$ depends on the initial choice of $\Phi$.  This group
determines the ``shape'' of the matrices $\theta$ and $\omega$ as well as the
projections that define the components of $P$, $Q_{12}$,
$Q_-$, $r$, and $s$.  However, by the law of inertia, changing $\Phi$ for
another non-degenerate symmetric bilinear form of the same signature amounts
only to re-indexing these equations.  Under such a change, the spaces of
invariants will be represented by different (but isomorphic) submodules of
$\mathfrak{gl}(n) \otimes \mathfrak{gl}(n)$.

\begin{proof}
Consider the second-order embeddable skewing map $\Delta:
(\mathcal{S}^* \oplus \mathbb{R}\Phi) \otimes \mathcal{S}^* \to
\mathfrak{gl}(n) \otimes (\mathcal{S}^*\wedge\mathcal{S}^*)$ defined by
$\Delta(C)(\omega\wedge\omega) = C(\omega)\wedge\omega$.   Aside from the
restricted domain, this is the same as $\bar\Delta$ that appears in
Lemma~\ref{kerDelta}, so it is injective.  
By the identifications in Section~\ref{infingeom}, any
$C$ decomposes as $C(\omega) = s(\omega)\Phi + r\Phi \omega \Phi +
Q_{1}(\omega) + Q_{2}(\omega) + Q_{-}(\omega)$ according to  
${(\mathbb{R}\Phi + \mathcal{S}^*)} \otimes \mathcal{S}^* = 
\mathbb{R}\Phi \otimes \mathcal{S}^* + \mathcal{Q}_{0} + \mathcal{Q}_{1} +
\mathcal{Q}_{2} + \mathcal{Q}_{-}$.
Thus, it is known that this space provides all the second-order
invariants of $\mathcal{B}$.  The only question is where the various
components appear in the structure equations.

Equation~(\ref{str2}) is an equality of matrices in the $Ad(G)$-module
$\mathfrak{gl}(n)$, so it can be interpreted as the three distinct equations
by projecting onto the $\mathcal{R}$, $\mathbb{R}I$, and
$\mathfrak{so}(n,\Phi)$ submodules.   In particular, the second-order invariants of $\mathcal{B}$ will appear in the
equations 
\begin{equation}
\begin{split}
\pi_\mathcal{R}(\Delta(C)(\omega\wedge\omega)) &= -\nabla(\tau)\\
\pi_{\mathbb{R}I}(\Delta(C)(\omega\wedge\omega)) &= \textstyle{\frac12} \mathrm{d}\lambda
- \textstyle{\frac1n}\tr( \tau \wedge\tau )\\
\pi_{\mathfrak{so}(n,\Phi)}(\Delta(C)(\omega\wedge\omega)) &=
-\mathrm{d}\varphi - \varphi\wedge\varphi - \tau \wedge\tau  +
\textstyle{\frac1n} \tr(\tau\wedge \tau).
\label{g}
\end{split}
\end{equation}

Suppose $V \subset \mathfrak{gl}(n) \otimes (\mathcal{S}^* \wedge
\mathcal{S}^*)$ is an irreducible component of the image of
$\Delta$.  If $(\pi_\mathcal{R}\otimes 1) (V) \neq 0$ and
$(\pi_\mathfrak{g}\otimes 1)(V) \neq 0$, then there is an isomorphism between
these two images.  In this case, the irreducible component of
$R(\omega\wedge\omega) + \tau\wedge\tau$ appearing as
$(\pi_\mathfrak{g}\otimes 1)(V)$ may be expressed as a multiple of the
irreducible component of $\nabla(\tau)$ that appears as $(\pi_\mathcal{R}
\otimes 1)(V)$.  If $(\pi_\mathcal{R}\otimes 1)(V)=0$, then
$V=(\pi_\mathfrak{g}\otimes 1)(V)$, so $V$ is an irreducible component of
$R(\omega\wedge\omega) + \tau\wedge\tau$ that is independent of
$\nabla(\tau)$.  Finally, if $(\pi_\mathfrak{g}\otimes 1)(V)=0$, then
$V=(\pi_\mathcal{R} \otimes 1)(V)$, so $V$ is an irreducible component of
$\nabla(\tau)$ that is independent of $R(\omega\wedge\omega) +
\tau\wedge\tau$.   Similarly, the image on $\mathfrak{g}$ can be projected onto the
$\mathbb{R}I$ and $\mathfrak{so}(n,\Phi)$ submodules.
Thus, the form of Equation~(\ref{eqnembedstr}) relies only
on the decomposition of the images of the projections of $\Delta(C)$.
The proof is completed by Lemmas~\ref{dbb0}, \ref{dbb1}, and \ref{dbb2},
below. 
\end{proof}

Note that $r$ and $s$ take values in irreducible $(Od^*(G)\otimes Od^*(G))$-modules, but $P$
and $Q$ are not irreducible for general $n$.
To see the syzygies of the invariants $P$, $Q$, $r$, and $s$, it is necessary
to differentiate once more.

\begin{thm}[The Structure Theorem, cont'd]
In the setting of Theorem~\ref{thmembedstr}, the following equations also hold 
for $\tau = P(\omega)$ and $\rho = r\Phi \wedge \Phi + Q_{12}(\omega) + Q_{-}(\omega)$.
\begin{equation}
\begin{split}
\mathrm{d}(s(\omega)) &=   \lambda \wedge s(\omega) -2 \tr(\Phi^{-1} \rho\wedge P(\omega)), \\
\nabla(\rho) &= s(\omega) \wedge \left( \varphi \Phi + \Phi \varphi^{\top} +
\tau \Phi + \Phi\tau^{\top}\right) - \tau \wedge \rho + \rho \wedge
\tau^{\top} + \frac2n \tr(\Phi^{-1} \rho \wedge \tau)\Phi.\\
&\phantom{=} \text{action by $Od^*$}
\end{split}
\end{equation}
\label{thmembedstr2}
\end{thm}

Finer resolution of these syzygies among the invariants could be obtained for
specific $n$ using Clebsch--Gordan relations.

\begin{proof}
Write $C(\omega) = s(\omega) \Phi + \rho$ for $\rho = r\Phi\omega\Phi + Q_{12}(\omega) +
Q_{-}(\omega)$ a semi-basic 1-form valued in $\mathcal{S}^*$.
By pulling back the final part of Equation~\ref{mc}, it must be
that 
\begin{equation}
0=
\mathrm{d}\left(s(\omega) \Phi + \rho\right) + \left(\varphi - \frac12 \lambda I +
\tau\right)\wedge(s(\omega) \Phi + \rho) - (s(\omega)\Phi + \rho)\wedge\left(\varphi
- \frac12 \lambda I+\tau\right)^{\top}.
\end{equation}
So, $\mathrm{d}s(\omega)$ is given by the $\Phi^{-1}$-trace of this equation,
and the components of $\mathrm{d}\rho$ are given by the projection of this
equation onto $\mathcal{S}^*$ by subtracting the $\Phi^{-1}$-trace part.
The terms $(\varphi + \tau)\wedge s(\omega)\Phi + \varphi \wedge \rho -
\frac12 \lambda  \wedge \rho$ and their transposes are $\Phi^{-1}$-traceless.
The component $-\frac12 \lambda \wedge s(\omega) \Phi$ and $\tau \wedge
\rho$ and their transposes have non-trivial $\Phi^{-1}$-trace.
\end{proof}

\begin{cor}
$P=0$ locally if and only if $P=Q_{12}=Q_{-}=s=0$ at a point.  In this case, $\mathrm{d}r =  2 r \lambda$.
\end{cor}
\begin{proof}
If $P=0$ locally, then $d(P(\omega))=0$, but the kernel of
$\pi_\mathcal{R} \circ \Delta$ is $Q_{0}$.  Therefore, $s=Q_{12}=Q_{-}=0$.
Write $\hat{\omega} = \Phi \omega \Phi$.
\begin{equation}
\begin{split}
0 &= \mathrm{d}(r  \hat{\omega} ) + r \theta\wedge \hat{\omega} - r\hat{\omega}\wedge \theta^{\top}\\
&= \mathrm{d} r  \wedge \hat{\omega} + r \Phi \left( 
- \lambda  \wedge \omega + \varphi^\top\wedge \omega - \omega \wedge \varphi
\right) \Phi + r \theta\wedge \hat{\omega} - r\hat{\omega}\wedge \theta^{\top}\\
&= \mathrm{d} r  \wedge \hat{\omega} + r \Phi \left( 
- \lambda  \wedge \omega + \varphi^\top\wedge \omega - \omega \wedge \varphi
\right) \Phi + r \theta\wedge \hat{\omega} - r\hat{\omega}\wedge \theta^{\top}\\
&= \mathrm{d}r \wedge \hat{\omega} - r \lambda \wedge \hat{\omega} 
-r \varphi \Phi \wedge \omega\Phi + r \Phi \omega \wedge \Phi \varphi^{\top}
+ - r \lambda \wedge \hat{\omega} + r \varphi \wedge\hat{\omega} 
-r \hat{\omega} \wedge \varphi^{\top}\\
&= (\mathrm{d}r - 2 r \lambda)  \wedge \hat{\omega}.
\end{split}
\end{equation}
Therefore, $\mathrm{d}r - 2 r \lambda \equiv 0$ modulo $\omega_{ij}$ for all
$i,j$.
\end{proof}

This helps us see that the invariant $r$ is somewhat spurious, in the sense
that it represents the scaling implicit in the identification $\mathcal{S}
\leftrightarrow \mathcal{S}^*$ or equivalently in the pairing
$\pair{\cdot,\cdot}$.  
Since this pairing gives pseudo-Riemannian metric over
$M$ that is only defined up to a conformal factor, the scalar curvature can be
varied freely.

\begin{cor}
Other than signature, there are no intrinsic invariants of Hessian PDEs in two variables.
All invariants arise from the particular embedding.
\label{no2}
\end{cor}
\begin{proof}
Consider an $Od(G)$-structure determined by $\Phi$ on a surface.   Then
$\mathcal{S} \cong \mathbb{R}^2$, so $\mathcal{S} \otimes \mathcal{S} \cong
\mathfrak{so}(2,\Phi) + \mathbb{R}I  + \mathcal{R}$, and $\mathcal{T}'=0$.  In
this case, $Q_{12}$, $Q_{+}$ and $s$ must also be identically zero, as they are derivatives
of $P = 0$.  The scalar function $r$ is still present, but Riemannian and
semi-Riemannian surfaces are
locally conformally flat, so $r$ can be normalized away.  Thus, there are no
intrinsic invariants of hyperbolic Hessian PDEs in two variables.  To glean
any information in this case, one must examine the extrinsic invariants of the
hypersurface $F^{-1}(0)$, as done in \cite{The2010}.  
\end{proof}

Here are the lemmas that provide the decompositions in the structure theorem.

\begin{lemma}
The kernel of $(\pi_{\mathbb{R}I} \otimes 1)\circ \Delta$ is $(\mathbb{R}\Phi
\otimes \mathcal{S}^*) + \mathcal{Q}_{+} \subset (\mathbb{R}\Phi \oplus
\mathcal{S}^*) \otimes \mathcal{S}^*$.
\label{dbb0}
\end{lemma}
\begin{proof}
Suppose that $0 = \pi_{\mathbb{R}I}(\Delta(C)(A,B)) = 
\frac1n \tr\left( C(A)B - C(B)A \right)$ for all $A, B$ in
$\mathcal{S}$.    Write $C(A) = C_0 + C_1$ where $C_0 \in \mathbb{R}\Phi^{-1}  \otimes
\mathcal{S}^*$ and $C_1 \in \mathcal{S}^* \otimes \mathcal{S}^*$. 
Clearly, any $C=C_0$ is in the kernel, since $\Phi$-tracelessness defines
$\mathcal{S}$.
So, $C=C_0+C_1$ is in the kernel if and only if $C_1$ is in the kernel, meaning $0 =
\frac1n\tr(C_1(A)B - C_1(B)A)$ for all $A,B \in \mathcal{S}$.  By writing
$C_1=\Phi C_1^*\Phi$, one sees that this is precisely the condition that
$C_1\in \mathcal{Q}_+$.
\end{proof}

\begin{lemma}
$\ker (\pi_\mathcal{R} \otimes 1) \circ \Delta$
is isomorphic to the 1-dimensional submodule $\mathcal{Q}_0$ of $\mathcal{S}^* \odot
\mathcal{S}^*$ that is given by the scalar form $\pair{\cdot, \cdot}$.  
\label{dbb1}
\end{lemma}
\begin{proof}
Write $C = Y\Phi$ for $Y \in (\mathbb{R}I + \mathcal{R}) \otimes
\mathcal{S}^*$.  Suppose $C$ is in the kernel, so for all $A,B \in
\mathcal{S}$,
\begin{equation}
\begin{split}
0 
&= \pi_\mathcal{R}\left( \Delta(C)(A,B)\right) \\
&= \frac12 \left(  C(A)B - C(B)A + \Phi ( C(A)B - C(B)A )^{\top} \Phi^{-1}
\right) -\frac1n\tr(C(A)B-C(B)A) I \\
&= \frac12 \left( \Phi Y(A) B - \Phi Y(B) A + \Phi B Y(A) - \Phi A Y(B)
\right) -
\frac1n \tr(\Phi Y(A) B - \Phi Y(B) A)I  \\
&= \Phi \left( \frac12 \left( Y(A) B -  Y(B) A + B Y(A) - A Y(B) \right) -
\frac1n \tr(\Phi Y(A) B - \Phi Y(B) A)\Phi^{-1} \right) \\
&=
\Phi \left( \bar{\delta}(Y)(A,B) - \frac1n \tr( \Phi \bar{\delta}(Y)(A,B)) \Phi^{-1}
\right)\\
&= \Phi \pi_\mathcal{S} ( \bar{\delta}(Y)(A,B)).
\end{split}
\end{equation}
So, either $Y \in \ker \bar\delta$ or $\bar{\delta}(Y)(A,B)$ is a multiple of
$\Phi^{-1}$.   By Lemma~\ref{kerdeltabar}, the former case implies that $C(A) = c \Phi A \Phi$.
By the proof of Lemma~\ref{PR}, the latter case means that
$e(Y) = \Phi^{-1} C$ projects non-trivially to $\mathcal{Q}_-$; however, the
actual pre-image via $\bar\delta$ of $\Phi^{-1}$ also projects non-trivially to
$\mathfrak{so}(n)\otimes \mathcal{S}^*$, so it is not contained in the domain
of $\Delta$.
\end{proof}

\begin{lemma}
$\ker ((\pi_{\mathfrak{so}(n,\Phi)}\otimes 1) \circ \Delta )= \ker (\Delta -
\Delta^{\top})$ is isomorphic to the $m$-dimensional irreducible representation $\mathbb{R} \Phi \otimes
\mathcal{S}^*$. 
\label{dbb2}
\end{lemma}

The proof of this lemma is a repeated use of Cartan's lemma, analogous to that
of Lemmas~\ref{kerDelta} and \ref{kerdeltabar} but more tedious.

\subsection{A Classifying Space for Second-Order PDEs}
Because the Theorems~\ref{thmembedstr} and \ref{thmembedstr2} provide
structure equations that are closed under exterior derivative and 
have finitely many structure coefficients, they fit into the framework of
Cartan's generalization of Lie's third fundamental theorem, which concerns the
existence and uniqueness of pseudo-groups with putative structure equations
\cite{Cartan1904}.
For a summary of the case needed here, see Appendix A of
\cite{Bryant2001a}.  A modern interpretation of this theorem arises in the
theory of groupoids and Lie algebroids \cite{Mackenzie2005}.  All of the
results here are standard consequences of the fact that the space of
differential invariants is finite-dimensional \cite{Olver1995}.  It is
worthwhile to compare this section to Sections~4 and 5 of \cite{Smith2009a}, which
only applies to integrable hyperbolic Hessian PDEs.   The presentation here
intentionally mirrors that one closely, but the theorems here apply more
generally to non-degenerate Hessian PDEs in any number of variables.

\begin{defn}
The notation $(\mathcal{B}, M, p)_\Phi$ denotes a smooth, embeddable $Od(G)$-structure over $M$ such that $M$ is
connected and such that $p \in M$.
\end{defn}

Connectedness is very important for these corollaries.

\begin{defn}
Let $\mathbf{K}_\Phi = \mathcal{T}' \oplus (\mathcal{Q}_{1} \oplus
\mathcal{Q}_{2}) \oplus \mathcal{Q}_{-} \oplus \mathcal{Q}_{0} \oplus
\mathcal{S}^*$, a vector space of dimension $m^2+\frac12m(m+1)-1$.
This space is called the ``classifying space'' for $Od(G)$-structures.
For an $Od(G)$-structure $\mathcal{B}$, let $\kappa:\mathcal{B}\to \mathbf{K}_\Phi$
denote the function $\kappa(b) = (P(b),Q_{12}(b),Q_{-}(b),r(b),s(b))$.
\end{defn}

Note that $\mathbf{K}_\Phi$ is an $O(n,\Phi)$-module by the appropriate actions on each
component.  Suppose that $\mathcal{T}'\cong \mathcal{Q}_{2}$ decomposes into irreducible $(Ad(G)\otimes
Od^*(G))$-modules or $(Od^*(G) \otimes Od^*(G))$-modules as
$\mathcal{Q}_{2,1}, \ldots, \mathcal{Q}_{2,q}$, and that $\mathcal{Q}_{-}$
decomposes into irreducible modules $\mathcal{Q}_{-,1}, \ldots
\mathcal{Q}_{-,w}$.  The component $\mathcal{S}^*$ is an
irreducible $Od^*(G)$-module, and the component $\mathbb{R}$ is trivial.  Then
\[\mathbf{K}_\Phi \cong (\mathcal{Q}_{2,1} \oplus \cdots \oplus \mathcal{Q}_{2,q})
\oplus  \mathcal{Q}_1 \oplus (\mathcal{Q}_{2,1} \oplus
\cdots \oplus \mathcal{Q}_{2,q}) \oplus ( \mathcal{Q}_{-,1} \oplus \cdots
\oplus \mathcal{Q}_{-,w}) \oplus \mathcal{Q}_{0} \oplus \mathcal{Q}_{1}\]
is a decomposition of classifying space into irreducible
$O(n,\Phi)$-modules.  
The infinitesimal scaling action acts on each component of $\mathbf{K}_\Phi$
as well, so there is a group action corresponding to $\mathfrak{g}$,
constructed analogously to Appendix~\ref{neg}. If the scaling action is removed by considering the projective
group, then each component becomes the corresponding projective space, in
which case the component $\mathcal{Q}_0\cong \mathbb{R}$ (where $r$ lives) vanishes to a point.

Implicit in Theorem~\ref{thmembedstr2} is a set of equations for
$\mathrm{d}\kappa$ that 
schematically looks like
\begin{equation}
\mathrm{d}\kappa_b = \mathrm{d} \begin{pmatrix}
P\\ Q_{12} \\ Q_{-} \\ r \\ s
\end{pmatrix}
= 
\begin{pmatrix}
\text{linear in $Q_{12},Q_{-},s$} &  \text{quadratic in $P$} \\ 
\text{quadratic in $P,Q_{12},Q_{-},r,s$} &  \text{linear in $Q_{12},Q_{-},r,s$} \\ 
\text{quadratic in $P,Q_{12},Q_{-},r,s$} &  \text{linear in $Q_{12},Q_{-},r,s$} \\ 
\text{quadratic in $P,Q_{12},Q_{-},r,s$} &  \text{linear in $Q_{12},Q_{-},r,s$} \\ 
\text{quadratic in $P,Q_{12},Q_{-},s$} &  \text{linear in $s$} \\ 
\end{pmatrix}
\begin{pmatrix}
\omega_{ij} \\
\theta_{ij} \\
\end{pmatrix}
= J(\kappa(b)) \begin{pmatrix}
\omega \\ \theta
\end{pmatrix}
\label{kappaJ}
\end{equation}
The matrix $J(K)$ given by these formulas is well-defined for any $K \in
\mathbf{K}_\Phi$.  The matrix has $\dim \mathcal{B} =n^2$ columns and $\dim\mathbf{K}_\Phi
= m^2 + \frac12m(m+1)-1$ rows, and its entries are algebraic functions of $K
\in \mathbf{K}_\Phi$ as given by a complete expansion of
Theorem~\ref{thmembedstr2}.
Because the equations in
Theorems~\ref{thmembedstr} and \ref{thmembedstr2}  are closed under exterior
derivative, the matrix $J$ defines the anchor map of a Lie algebroid over $\mathbf{K}_\Phi$. This Lie
algebroid can be integrated to a smooth groupoid over $\mathbf{K}_\Phi$
\cite{Stefan1980,Mackenzie2005}.

\begin{lemma}
The singular distribution on $\mathbf{K}_\Phi$ defined by the columns of the
matrix $J$ is integrable, providing a singular foliation of $\mathbf{K}_\Phi$
into leaves that are submanifolds. That is, for any $K \in \mathbf{K}_\Phi$, there
exists a unique submanifold $\mathcal{O}_J(K)$ $\mathbf{T}_K \mathcal{O}_J(K)
= \mathop{\mathrm{range}}J(K)$.
\label{leaves}
\end{lemma}

Through each point $K \in \mathbf{K}_\Phi$, there passes an orbit,
$\mathcal{O}_J(K)$, and each orbit is a submanifold of the base. These orbits
may be regarded as the leaves of a singular foliation of $\mathbf{K}_\Phi$
such that $\mathbf{T}_K\mathcal{O}_J(K)$ is spanned by the columns of $J(K)$.
The next two corollaries are standard consequences of Cartan's structure
theorem, and are proven by considering the symmetry algebra of each leaf and
applying the Cartan--K\"ahler theorem to construct solutions.

\begin{cor}[Existence]
For any non-degenerate $\Phi$ and for any choice of 
$K \in \mathbf{K}_\Phi$, there exists
$(\mathcal{B},M,p)_\Phi$ such that $\kappa(b)=K$ for some $b \in \mathcal{B}_p$.
Moreover, this structure is analytic.  The map $\kappa$ is a submersion to the
leaf $\mathcal{O}_J(K)$.
\end{cor}

\begin{cor}[Uniqueness]
Suppose that $(\mathcal{B},M,p)_\Phi$ and $(\mathcal{B}',M',p')_\Phi$ are two 
$Od(G)$-structures such that $\kappa(b)=\kappa'(b')$ for some $b \in \mathcal{B}_p$ and $b' \in \mathcal{B}_{p'}$.
Then there exist neighborhoods $U$ of $p$ and $U'$ of $p'$ such that
$\mathcal{B}_U$ and $\mathcal{B}'_{U'}$ are isomorphic as $Od(G)$-structures.
\end{cor}
Note that, \emph{a priori}, the local isomorphism only holds when comparing two structures using the
same $\Phi$; however, the law of inertia allows one to construct a bundle
isomorphism between $(\mathcal{B},M,p)_\Phi$ and
$(\mathcal{B}',M',p')_{\Phi'}$ as long as $\Phi$ and $\Phi'$ have the same
signature.   

Lemma~\ref{embeddable} shows that the existence and uniqueness theorems apply
to local non-degenerate Hessian PDEs, too:
\begin{cor}
For a fixed $\Phi$ and for any choice of $K \in
\mathbf{K}_\Phi$, there exists an analytic
function $F$ defined in a neighborhood of the origin of
$\Lambda^o$ such that the $Od(G)$-structure induced on
$F^{-1}(0)$ takes the chosen value of $K$ in the fiber over the origin.
The Taylor series of this function at the origin is uniquely determined by
the $G$-orbit of $K$.
\end{cor}

This construction is essentially what is done in Section~6 of
\cite{Smith2009a} for several examples in the integrable case for $n=3$.
For any $n$, the flat structure yields the PDE
$0=\sum_{ij}\Phi_{ij} \frac{\partial^2 z}{\partial x^i \partial x^j}$.

\begin{defn}
$(\mathcal{B},M,p)_\Phi$ is said to ``represent $K$'' if $K = \kappa(b)$ for some
$b \in \mathcal{B}_p$.
\end{defn}


\begin{defn}[Leaf-equivalence]
$(\mathcal{B}_0,M_0,p_0)_\Phi$ and $(\mathcal{B}_k,M_k,p_k)_\Phi$ are said to be
\textbf{leaf-equivalent} if there exist finite sequences
$\{(\mathcal{B}_i,M_i,p_i)_\Phi\}$ and $\{K_i\}$
with $1 \leq i \leq k{-}1$ such that $(\mathcal{B}_i,M_i,p_i)_\Phi$ and
$(\mathcal{B}_{i+1},M_{i+1},p_{i+1})_\Phi$ both represent $K_i$ for $0 \leq i
\leq k{-}1$. 
\end{defn}

The term ``leaf-equivalence'' arises from the leaves of the singular foliation
of $\mathbf{K}_\Phi$ from Lemma~\ref{leaves}.  Again, a standard argument for
Lie pseudo-groups shows that these leaves separate all
possible $(\mathcal{B},M,p)_\Phi$'s into equivalence classes by the value of
$K$.   
\begin{thm}[Leaf-equivalence]
$(\mathcal{B},M,p)_\Phi$ and $(\hat{\mathcal{B}}, \hat M, \hat p)_\Phi$ are
leaf-equivalent if and only if $\mathcal{O}_J(\mathcal{B}) =
\mathcal{O}_J(\hat{\mathcal{B}})$.  Moreover, $\kappa:\mathcal{B} \to
\mathcal{O}_J(\mathcal{B})$ is a submersion.  \label{leafequiv} 
\end{thm}

Because $\kappa$ is a submersion, the leaves have dimension at most $\dim
\mathcal{B} = n^2$, which is much smaller than $\dim \mathbf{K}_\Phi$.
Identifying these leaves explicitly is potentially an extremely challenging
task, but it is ultimately the way towards a thorough understanding of
non-degenerate Hessian PDEs.  For the special case of integrable hyperbolic
Hessian PDEs in three variables, this was accomplished due to a small miracle
(Lemma 5.1 of \cite{Smith2009a}).

\subsection{The Case of Three Variables}\label{3}
Consider the hyperbolic signature $(2,1)$ for $n=3$.  The base manifold $M$ has dimension 5
and $\Lambda^o$ has dimension 6.  If $\Phi$ is chosen to be
\[ 
\Phi = 
\begin{pmatrix}
0 & 0 & \textstyle{-\frac12} \\
0 & 1 & 0 \\
\textstyle{-\frac12} & 0 & 0
\end{pmatrix},
\]
then $\mathcal{S}$ is the vector space of matrices of the form
\[ 
\begin{pmatrix}
a_{-4} & a_{-2} & a_{ 0} \\
a_{-2} & a_{ 0} & a_{2} \\
a_{0} & a_{ 2} & a_{ 4}
\end{pmatrix}, a_{-4}, a_{-2}, a_{0}, a_{2}, a_{4} \in \mathbb{R}.\]
The Lie algebra $\mathfrak{so}(n,\Phi)$ is isomorphic to
$\mathfrak{sl}_2(\mathbb{R})$.  The finite-dimensional irreducible
representations of $\mathfrak{sl}_2(\mathbb{R})$ are given by the action of
$-x \frac{\partial}{\partial y}$, $y \frac{\partial}{\partial x}$, and
$x\frac{\partial}{\partial x} - y\frac{\partial}{\partial y}$ on
$\mathcal{V}_r$, the vector space of degree-$r$ homogeneous polynomials in $x$
and $y$.   
These and the scaling action correspond to the generators of the Lie algebra $od(\mathfrak{g})$.
\begin{equation}
\begin{pmatrix}
0 & 0 & 0 \\
-4 & 0 & 0 \\
0 & -2 & 0 
\end{pmatrix},
\begin{pmatrix}
0 & 2 & 0 \\
0 & 0 & 4 \\
0 & 0 & 0 
\end{pmatrix},
\begin{pmatrix}
4 & 0 & 0 \\
0 & 0 & 0 \\
0 & 0 & -4 
\end{pmatrix},
\begin{pmatrix}
-1/2 & 0 & 0 \\
0 & -1/2 & 0 \\
0 & 0 & -1/2 
\end{pmatrix}.
\end{equation}

The intersection of the Veronese cone with $\mathcal{S}$ is all matrices of
the form 
\[
\begin{pmatrix}
s^4 & s^3t & s^2t^2 \\
s^3t & t^2s^2 & st^3 \\
s^2t^2 & st^3 & t^4
\end{pmatrix}, s,t \in \mathbb{R}.\]
It is easy to see that if $\mathcal{S}$ is identified with
$\mathcal{V}_4 = \{ a_{-4} x^4 +  a_{-2} 4x^3y + a_{0} 6x^2y^2 + a_{2} 4xy^3 + a_{4} y^4 \}$,
then the intersection of the Veronese cone with $\mathcal{S}$ is the
rational normal cone $\{ (sx+ty)^4: s,t\in\mathbb{R} \}$, which has symmetry group $GL(2)$.
Hence, a hyperbolic $Od(G)$-structure in $n=3$ variables corresponds to a
$GL(2)$-structure of degree 4, as seen in \cite{Ferapontov2009} and \cite{Smith2009a}.
The identification of $\mathcal{S} = \mathcal{V}_4$ above is precisely the
identification that was used in \cite{Smith2009a} to reconstruct integrable
Hessian PDEs from the structure equations for 2,3-integrable
$GL(2)$-structures.

So, examining the structure theorem, we arrive at the following irreducible decompositions:
\begin{equation}
\begin{split}
\mathfrak{g} &\cong \mathcal{V}_0 \oplus \mathcal{V}_2\\ 
\mathcal{Q}_0 &\cong \mathcal{V}_0\\
\mathcal{Q}_1 &\cong \mathcal{V}_4\\
\mathcal{Q}_2 &\cong \mathcal{V}_8\\
\mathcal{Q}_{-}&\cong \mathcal{V}_2 \oplus \mathcal{V}_6\\
\mathbf{K}_\Phi &\cong (\mathcal{V}_8) \oplus (\mathcal{V}_4 \oplus \mathcal{V}_8)
\oplus (\mathcal{V}_2 \oplus \mathcal{V}_6) \oplus \mathcal{V}_0 \oplus
\mathcal{V}_4 \cong \mathbb{R}^{39}
\end{split}
\end{equation}

\begin{thm}
For a hyperbolic $Od(G)$-structures in $n=3$ variables or equivalently
for $GL(2)$-structures of degree 4, the condition of 2-integrability  (see
Section~\ref{int}) is equivalent to the condition of embeddability.
\end{thm}
\begin{proof}
In \cite{Smith2009a}, it is shown that a generic $GL(2)$-structure of degree 4
is 2-integrable if and only if its torsion only takes values in the
irreducible representation $\mathcal{V}_8$.  Compare to
Lemma~\ref{fullytraceless}.
\end{proof}

Thus, the theory of embeddable hyperbolic $Od(G)$-structures in $n=3$
variables is equivalent to the study of 2-integrable $GL(2)$-structures of
degree 4.  See Corollary~3.2 of \cite{Smith2009a}, where structure equations
appear that are equivalent to Theorem~\ref{thmembedstr} above, albeit with a
different collection of projections.
To see the relationship between the projections, note that 
the decomposition $V \otimes V = \mathbb{R}\Phi^{-1} + \mathfrak{so}(3) +
\mathcal{S}$ corresponds to the decomposition $\mathcal{V}_2 \otimes
\mathcal{V}_2 = \mathcal{V}_0 \oplus \mathcal{V}_2 \oplus \mathcal{V}_4$.
The coefficients of the Clebsch--Gordan pairings for $SL(2)$ $\pair{ v, w}_0$, $\pair{v ,w}_1$, and
$\pair{v,w}_2$ are scalar multiples of the 
coefficients of $\pi_\mathcal{R}(v^{\top} w)$, $\pi_{\mathfrak{so}(3)}(v^{\top} w)$ and
$\pi_{\mathbb{R}\Phi^{-1}}(v^{\top}w)$, respectively.

The leaves in $\mathbf{K}_\Phi$ can have dimension no greater than $3^2=9$.
In the 2,3-integrable case, all of the second-order invariants become functions of
$P$, and there are indeed leaves in $\mathcal{V}_8 = \mathcal{Q}_2$ of maximum
dimension nine.

\subsection{The Case of Four Variables}\label{4}
The Lie algebra $\mathfrak{so}(3,1)$ is isomorphic to
$\mathfrak{sl}_2(\mathbb{R}) \times \mathfrak{sl}_2(\mathbb{R})$, and the
finite-dimensional representations are given by products of homogeneous polynomials 
$\mathcal{V}_{p,q} = \{ f(x,y)g(x',y'), \deg f = p, \deg g = q\}$, which has
dimension $(p+1)(q+1)$.  
The standard representation on $V$ is denoted by $\mathcal{V}_{1,1}$, and the 
$Od$ representation on $\mathcal{S}$ is $\mathcal{V}_{2,2}$.
Using the well-known decomposition for these representations (see
\cite{Cartan1981, Gelfand1958}) we can compute
\begin{equation}
\begin{split}
\mathfrak{g} &\cong \mathcal{V}_{0,0} \oplus \mathcal{V}_{0,2} \oplus \mathcal{V}_{2,0}\\
\mathcal{Q}_{0} &\cong \mathcal{V}_{0,0}\\
\mathcal{Q}_{1} &\cong \mathcal{V}_{2,2} \\
\mathcal{Q}_{2} &\cong \mathcal{V}_{0,4} \oplus \mathcal{V}_{4,0} \oplus \mathcal{V}_{4,4}\\
\mathcal{Q}_{-} &\cong \mathcal{V}_{2,0} \oplus \mathcal{V}_{0,2} \oplus
\mathcal{V}_{2,4} \oplus \mathcal{V}_{4,2}\\
\mathbf{K}_\Phi &\cong \mathbb{R}^{90}.
\end{split}
\end{equation}
The leaves in $\mathbf{K}_\Phi$ can have dimension no greater than $4^2=16$,
but $P$ takes values in a sum of irreducible representations of dimensions 
five, five, and 25.

\section{Secant submanifolds and Hydrodynamic Integrability}
\label{int}
When studying PDEs, one often considers the question of integrability; that
is, when can one construct ``many'' exact solutions making clever use of
characteristics or conservation laws?
This section is a summary of the approach that has been championed by
Tsarev, Ferapontov and their many collaborators over the past two decades 
\cite{Tsarev1990,Tsarev1993, Tsarev2000, Ferapontov2002,
Ferapontov2003, Ferapontov2004, Burovskiy2008, Ferapontov2009, Doubrov2009}.
(In fact, this article is the result of an effort to understand their approach
to integrability.  It is a happy accident that this effort lead to broader statements about general,
non-integrable PDEs.)  This section exists simply to 
demonstrate that $Od(G)$-structures provide a convenient geometric framework to
investigate integrability; Part II will be dedicated to that investigation.

For scalar PDEs in three or more variables, a popular notion of integrability
seems to be tied to the existence of hydrodynamic reductions.

\begin{defn}[Hydrodynamic Reduction]
Consider a PDE $F^{-1}(0) \subset J^2(\mathbb{R}^n,\mathbb{R})$.  A $k$-parameter hydrodynamic
reduction for $F^{-1}(0)$ is a pair of maps $(R,Z)$ of
the form
\begin{equation}
\mathbb{R}^n \overset{R}\longrightarrow \mathbb{R}^k
\overset{Z}\longrightarrow J^2(\mathbb{R}^n,\mathbb{R})
\end{equation}
such that 
\begin{enumerate}
\item $R$ is a submersion, and $Z$ is an immersion;
\item $N=Z(\mathbb{R}^k)$ is a $k$-dimensional submanifold of $F^{-1}(0)$ and
is an integral of the contact system on $J^2$;
\item for $l=1,\ldots, k$ there exist functions $\lambda^l:\mathbb{R}^k \to \mathbb{R}^n$ such that 
$\frac{\partial}{\partial x^i}R^l = \lambda^l_i(R) \frac{\partial}{\partial
x^1}R^l$; 
\item there exist $\Gamma^l_b:\mathbb{R}^n \to \mathbb{R}$ such that 
$\frac{\partial}{\partial R^b}\lambda^l_i = (\lambda^l_i
-\lambda^b_i)\Gamma^l_b$.
\end{enumerate}
\end{defn}

This definition is built upon the notion of constructing systems of
conservation laws that foliate the hypersurface $F^{-1}(0) \subset J^2$.
Condition (1) is a simple non-degeneracy assumption, for if $R$ were not a
submersion, then one would simply reduce the dimension $k$ to match the image
of $R$.  Condition (2) essentially means that $N$ can be treated as an
intermediate solution of $F$.  Conditions (3) and (4) may seem cumbersome, but they are perfect for reducing
$F=0$ to a system of coupled first-order PDEs in the $\lambda$'s.  In fact,
conditions (3) and (4) are familiar from the definition of systems of
conservation laws in $(1{+}1)$ variables that are rich or semi-Hamiltonian.   The
vectors $\{ \lambda^1, \ldots \lambda^k\}$ can be interpreted as the
characteristic speeds of a traveling wave within $F^{-1}(0)$.  When
these reductions exist, the reduced systems can be used to construct a
solution to the original equation using the generalized hodograph method
\cite{Tsarev1990}.  If this can be done in ``many'' ways, then the PDE is called
integrable.

\begin{defn}[Integrability for PDEs]
A PDE $F=0$ in $n\geq 3$ independent variables is integrable if, for all 
$k=1, \ldots, n$, there are infinitely many $k$-parameter hydrodynamic reductions of $F=0$,
and this collection is parametrized by ${k(n-2)}$ functions of one variable.
\end{defn}
In the references cited above, the number $k(n-2)$ is expected to be the
maximal possible Cartan character in the generic case.  This parametrization
will be discussed in greater detail in Part II; a more geometric definition is
provided below.

Of course, these definitions can be extended in sensible ways for PDEs of
higher order or in more dependent variables, but we focus on Hessian scalar
PDEs here.  In this case, the map $Z$ of a hydrodynamic reduction takes values
in $N = Z(\mathbb{R}^k) \subset F^{-1}(0) \subset \Lambda^o$.  Recall that
$\mathrm{d}U$ is a flat coframing on $\Lambda^o$ from Section~\ref{back}.

\begin{lemma}[e.g., \cite{Ferapontov2009}]
Suppose $F=0$ is a Hessian PDE in $n \geq 3$ variables such that
$\mathrm{d}F_U
= \sum_{ij} \Phi(U)_{ij} a(U)_{ij}$ for a flat $(V\odot V)$-valued coframe
$a=a^{\top}$ on $\Lambda^o$.
Let $(R,Z)$ be a $k$-parameter hydrodynamic reduction of $F=0$ with
$N=Z(\mathbb{R}^k)$.
Then $\mathbf{T}N$ is everywhere spanned by $k$ tangent vectors $\{A^1,
\ldots, A^k\}$ such that $a(A^j)$ lies in the intersection of the hyperplane
$\Phi(U)^\perp \subset V \odot V$ with the Veronese cone in $V \odot V$.
\label{obs}
\end{lemma}
\begin{proof}
Following
\cite{Ferapontov2009}, we observe
that (if $Z_{ij}$ is to be the Hessian matrix of a smooth function $z$) 
\begin{equation} 
\frac{\partial Z_{ij}}{\partial x^k} = \frac{\partial Z_{ik}}{\partial x^j} =
\frac{\partial Z_{jk}}{\partial x^i}, \end{equation}
which implies that 
\begin{equation} 
0= \frac{\partial Z_{ij}}{\partial x^1} -
\frac{\partial Z_{1i}}{\partial x^j} 
=
\sum_l \frac{\partial Z_{ij}}{\partial R^l} \frac{\partial R^l}{\partial x^1}-
\sum_l \frac{\partial Z_{1i}}{\partial R^l} \frac{\partial R^l}{\partial x^j} 
=
\sum_l \left( \frac{\partial Z_{ij}}{\partial R^l} -
\frac{\partial Z_{1j}}{\partial R^l} \lambda^l_j\right) 
\frac{\partial R^l}{\partial x^1}.
\end{equation}
Therefore, because $R(x)$ is not constant, we have
\begin{equation}
\frac{\partial}{\partial R^l} \lhk \mathrm{d}Z_{ij} = 
\frac{\partial Z_{ij}}{\partial R^l} = \frac{\partial Z_{11}}{\partial R^l}
\lambda^l_i \lambda^l_j.
\end{equation}
Thus, the image of $\mathrm{d}Z$ is a rank-one symmetric matrix in a flat
coframe.

Let $v^l =\lambda^l \circ Z^{-1}$ for $l=1,\ldots,k$ and $A^l = (v^l)^{\top}
v^l$.
\end{proof}

Thus, hydrodynamic reductions and integrability are intimately tied to the
Veronese variety and its intersection with $\mathbf{T}(F^{-1}(0)) = \ker
\mathrm{d}F$.  With this observation in mind, there are obvious analogous
notions for $Od(G)$-structures.  These definitions are designed to admit
analysis using the Cartan--K\"ahler theorem.

\begin{defn}[Secants]
Suppose $\mathcal{B} \to M$ is an $Od(G)$-structure (not necessarily embeddable).
A $k$-dimensional subspace $E^k \subset \mathbf{T}_p M$ is
called $k$-secant if there exists $b \in \mathcal{B}_p$ such that $b(E) \subset
\mathcal{S}$ is the span of $\{A^1, \ldots, A^k\}$ such that
each $A^l$ is a symmetric rank-one matrix.  That is, $A^l = \ver_2(v^l ) =
(v^l)^{\top}(v^l)$ and $v^l \Phi (v^l)^{\top} = 0$.  A $k$-dimensional
submanifold $N \subset M$ is called $k$-secant if the sub-tangent space
$\mathbf{T}_{p} N$ is $k$-secant for all $p \in N$.
\end{defn}
Since the null vectors $v^1, \ldots, v^k$ are independent, any $k$-secant
subspace $E$ contains exactly $k$ distinct $(k{-}1)$-secant subspaces.  

\begin{defn}[Integrability for Structures]
An $Od(G)$-structure $B \to M$ is called $k$-integrable if, for any 
$k$-secant subspace $E \subset \mathbf{T}_p M$, there exists a $k$-secant
submanifold $N$ with $\mathbf{T}_p N = E$.
If $B\to M$ is, for example, 2-integrable and 3-integrable, this property is
abbreviated as 2,3-integrable. ``Integrable'' is shorthand for
``$2,\ldots,n$-integrable.''
\label{defint}
\end{defn}

Note that 1-integrability always holds, as it describes the existence of
the flow of a vector field. 

\begin{thm}[2-Integrability]
For any $n$, every hyperbolic, embeddable $Od(G)$-structure $\mathcal{B} \to M$ is 2-integrable.
In this case, The solution 2-secant submanifolds are locally parametrized by
$2(n-2)$ functions of one variable.  This is true in the smooth category.
\label{2integrable}
\end{thm}

\begin{proof}
Consider the hyperbolic bilinear form
$x_1x_n = \sum_{k=2}^{n-1} (x_k)^2$, which corresponds to the symmetric matrix
\begin{equation}
\Phi= \begin{pmatrix}
0 & \cdots & -1/2 \\
\vdots & I_{n-2} & \vdots \\
-1/2 &  \cdots & 0
\end{pmatrix}
\end{equation}
Suppose $E \subset \mathbf{T}_p M$ is a bi-secant plane.  
In a neighborhood of $p$, choose a section $b:M \to \mathcal{B}$.  It must be
that $b(E)$ is spanned by two rank-one symmetric matrices that are in the
Veronese image of null vectors.  
For any two null vectors in $V$, there exists an element of $G$ that 
moves these null vectors to the null vectors $y^1=(1, 0, \ldots, 0)$ and $y^2=(0,
\ldots, 0,1)$.  So, one may use an $Od(G)$ frame adaptation to assume that
$b(E)$ is the span of $Y^1=\ver_2(1,0,\ldots,0)$ and $Y^2=\ver_2(0,\ldots,0,1)$. 
To prove the theorem, one must find the conditions on $\mathcal{B}$ that allow
an arbitrary bi-secant plane $E \in Gr_2(\mathbf{T}_pM)$ to be extended to a
bi-secant surface $N \subset M$.   Let $\tilde{E}=b_*(E)$ denote the
``lifted'' image of $E$ in $\mathbf{T}_{b(p)}\mathcal{B}$.  Then
$\omega_{11}\wedge \omega_{nn}|_{\tilde{E}} \neq 0$ and the annihilator of
$\tilde{E}$ is $\{ \omega_{ij},  (i,j) \neq (1,1),(n,n)\}$.  Let $\mathcal{I}$
denote the differential ideal generated by these 1-forms with the independence
condition $\omega_{11}\wedge\omega_{nn}$.  It suffices to
prove the existence of integral manifolds of this differential ideal.

Recall that $\omega_{ij} = \omega_{ji}$ and $\omega_{1n} = \sum_{k=2}^{n-1}
\omega_{kk}$.  By the first-order structure equations, $\mathrm{d}\omega = (\theta +
\tau)^{\top}\wedge\omega - \omega\wedge(\theta+\tau)$, it is clear that
$\mathrm{d}\omega_{ij} \equiv 0$ unless $i$ or $j$ equals $1$ or
$n$.  Also, $\omega_{1n} = \omega_{22} + \cdots + \omega_{n-1,n-1}$ implies
that $\mathrm{d}\omega_{1n}\equiv 0$. 
Thus, the differential generators are 
\begin{equation}
\mathrm{d}\begin{pmatrix}
\omega_{1i} \\
\omega_{in} 
\end{pmatrix}
\equiv 
\begin{pmatrix}
\theta_{1i} + \tau_{1i} & 0  \\
0 & \theta_{ni} + \tau_{ni} 
\end{pmatrix}\wedge\begin{pmatrix}
\omega_{11} \\ \omega_{nn}
\end{pmatrix},\  i=2,\ldots,n{-}1.
\end{equation}
The $\theta$'s that appear here are independent, so any
torsion terms arising from $\tau=P(\omega)$ are always absorbable, so Cartan's
test shows that the system is involutive with solutions depending on
$s_1=2(n-2)$ functions of one variable.
Moreover, this is a hyperbolic linear Pfaffian system in the sense of
\cite{Yang1987}, so the Cartan--K\"ahler theorem applies in the $C^\infty$
category.
\end{proof}

Theorem~\ref{2integrable} is already well-known from the perspective of PDEs, but its proof provides a model for how to approach
$k$-integrability in general.   Either of Corollary~\ref{no2} or
Theorem~\ref{2integrable} shows why hydrodynamic integrability is a trivial
concept for PDEs in $n=2$ variables.  The case of integrability for hyperbolic
Hessian PDEs in $n=3$ variables is detailed in \cite{Ferapontov2009} and
\cite{Smith2009a}.  The case of integrability (equivalently,
2,3-integrability) for symplectic Monge--Amp\`ere equations in $n=4$
variables is detailed in \cite{Doubrov2009}.  In both cases, the geometry
induced from the symmetries of the Veronese cone is used to classify the
integrable equations.  Thus, the extrinsic geometry is tied to the
$Od(GL(n))$-structure over $\Lambda^o$ and the intrinsic geometry is tied to
the induced $Od(G)$-structure on $F^{-1}(0)$.
The mostly-open case of $k$-integrability for $k\geq 3$ and $n\geq 3$ will be studied in Part II.

\section{Conclusion}
A Hessian partial differential equation $M=F^{-1}(0) \subset
\mathrm{Sym}^2(\mathbb{R}^n)$ of any fixed non-degenerate signature in any
number, $n$, of variables admits a geometry, called an embeddable
$Od(G)$-structure, with structure equations that are complete at second-order.  The natural
notion of integrability from PDE theory translates to a natural notion of
geometric integrability for these structures.   Up to conformal factors, the
fiber group of this $Od(G)$-structure is a subgroup of the orthogonal group
$O(n(n+1)/2-1,\pair{\cdot,\cdot})$ for a pseudo-Riemannian structure on $M$,
and it is a representation of the orthogonal group $O(n, \mathrm{d}F)$.  The
structure equations are easy to write down in any dimension, and the structure
functions take values in a finite-dimensional classifying space
$\mathbf{K}_\Phi$ that is given by the 1-jet of a first-order invariant $P \in
\mathcal{Q}_{2}$ along with a scalar curvature $r$ that reflects the conformal
factor of the pseudo-Riemannian structure.

Using the standard theory of Lie pseudo-groups of finite type, several
conclusions can be drawn immediately.  To each point in $\mathbf{K}_\Phi$
there is an associated Hessian PDE, which is locally unique.  There is a
singular foliation of $\mathbf{K}_\Phi$ that separates connected
embeddable $Od(G)$-structures into equivalence classes, and the corresponding moduli
space depends only on the signature of the leading symbol of $F$.  Because the
leaves of $\mathbf{K}_\Phi$ have high co-dimension in the case $n \geq 3$, there are
infinitely-many such equivalence classes. 

Moreover, because any $CSp(n)$-invariant classification of non-degenerate
Hessian PDEs induces a contact-invariant classification of second-order PDEs
that have locally non-degenerate leading symbol, the moduli space of Hessian
PDEs defined by the singular foliation of $\mathbf{K}_\Phi$ also provides a
classification of all such second-order PDEs; however, the full
contact-invariant classification of second-order PDEs will be much finer in
general.

While these are standard results from the theory of Lie pseudo-groups of
finite type, the explicit description of these structures should aid the
investigation of integrability and help to explain the increasing complexity
of PDEs in high dimensions and help identify special sub-classes of PDEs.

\newpage

\appendix

\section{Scaling and the $Od_g$ action}
\label{neg}
Let $V$ be a vector space of dimension $n$ over $\mathbb{R}$ or
$\mathbb{C}$ with the standard basis.  Let $V \odot V$ be identified with the
vector space of symmetric $n \times n$ matrices.

For any subgroup $PH$ of $PGL(n)$, define the representation of $POd$ on
$\mathbb{P}(V \odot V)$ by $POd_{[h]}(A) = [h^{\top} A h]$.  This
representation is faithful.

For $g \in GL(n)$ and $A$ a symmetric matrix, consider the action $Od_g(A) =
g^{\top} A g$. This action does not describe a faithful representation, as $Od_g = Od_{-g}$;
however, the infinitesimal action describes a faithful representation of
$\mathfrak{gl}(n)$:
\begin{lemma}
\label{injective}
Define $od:\mathfrak{gl}(n) \to \mathfrak{gl}(V \odot V)$ by $od_X(A) =
X^{\top} A + A X$.   Then $\ker od = 0$.
\end{lemma}

\begin{proof}
Suppose that $X \in \mathfrak{gl}(n)$ is such that $0 = X^{\top} A + AX$ for
all symmetric matrices $A$.  If $A$ is invertible, then $A$ represents a
non-degenerate symmetric bilinear form, and $X^{\top}$ is a matrix that is 
skew with respect to $A$.
Therefore, $X^{\top}$ lies in the intersection over all Lie algebras of the form
$\mathfrak{so}(n,A) = A\cdot \mathfrak{so}(n)$.  
\end{proof}

Over $\mathbb{C}$, the $Od_g$ action for $g \in GL(n)$ is transitive on
the Veronese cone in $\mathrm{Sym}^2(\mathbb{C}^n)$.  However, over
$\mathbb{R}$ the scaling by $-1$ is never possible through the $Od$ action, as
$Od_{\lambda I}(A) = \lambda^2 A$.  This annoyance can be dealt with in three
ways. 

First, one could consider only the action of $[g] \in PGL(n)$, and then the
target $[g^{\top} A g]$ is a representative in $PGL(n(n+1)/2)$ without
ambiguity.  This representation, called $POd(PGL(n))$, is faithful; however, it
disguises the structure equations that arise in the equivalence problem.


Alternatively, observe that the infinitesimal action $od_X$ actually does
allow for arbitrary scalings, as $X= -\frac12\lambda I$ shows.  So, 
define the group $Od(GL(n))$ as the collection of actions $\{ A \mapsto
\lambda g^{\top} A g, g \in GL(n), \lambda =  \pm 1 \}$.  This group
$Od(GL(n))$ is not really a representation of $GL(n)$, rather it can be
described as the semi-direct product of the faithful representation $POd(PGL(n))$ and the
one-dimensional scaling group $\mathbb{R}^\times$.  Note that the Lie algebra of
$Od(GL(n))$ is isomorphic to the Lie algebra of $\mathfrak{gl}(n)$, though the
scaling action is halved.

Another group of particular interest in this article is $G=CO(n,\Phi)$ for a
non-degenerate symmetric bilinear form $\Phi$.  The group $Od(G)$ is defined
analogously as $\{ A \mapsto \lambda g^{\top} A g,\ g \in CO(n,\Phi), \lambda
= \pm 1 \}$.  The Lie algebra of this group is $\{ A \mapsto X^{\top}A + AX +
\lambda A, X \in \mathfrak{so}(n,\Phi), \lambda \in \mathbb{R}\}$.  Of course,
this is isomorphic to $\mathfrak{g}=\mathfrak{so}(n,\Phi) + \mathbb{R}I$.
Thus, when considering irreducible $Od(G)$-modules, one needs to examine the
irreducible representations of $\mathfrak{so}(n,\Phi)$.

Finally, if both of these approaches are distasteful, then one could simply pretend that this is a
study of non-degenerate Hessian PDEs over the complex numbers and forget this
detail entirely.

\newpage

\bibliographystyle{amsalpha}
\bibliography{hydro}

\end{document}